\documentclass[reqno]{amsart}

\usepackage{amsfonts}
\usepackage{graphicx}
\usepackage{amsfonts,amsmath, amssymb}

\usepackage[colorlinks,linkcolor=blue,anchorcolor=red,citecolor=blue]{hyperref}
\usepackage{geometry}
\usepackage{a4}

\usepackage{marginnote}
\usepackage{color}
\usepackage{mathtools}
\usepackage{cancel}
\usepackage{cite}
\numberwithin{equation}{section}

\newcommand{\R}{\mathbb{R}}

\newtheorem{theorem}{Theorem}[section]

\newtheorem{lemma}[theorem]{Lemma}
\newtheorem{proposition}[theorem]{Proposition}

\newtheorem{definition}[theorem]{Definition}

\def\g{\gamma}
\def\d{\delta}

\def\f{\frac}

\def\b{\bar}

\usepackage{tikz}

\newcommand{\dd}{{\rm d}}

\renewcommand{\S}{\mathbb{S}}

\newcommand{\Fi}{\mathbf{1}}

\newcommand{\CD}{\mathcal {D}}

\newcommand{\CV}{\mathcal{V}}

\newcommand{\CS}{\mathcal{S}}

\newcommand{\om}{\omega}
\newcommand{\la}{\lambda}

\newcommand{\si}{\sigma}
\newcommand{\pa}{\partial}

\newcommand{\eps}{\epsilon}

\newcommand{\Ga}{\Gamma}

\makeatletter
\@namedef{subjclassname@2020}{%
  \textup{2020} Mathematics Subject Classification}
\makeatother

\begin{document}
\title[A moving boundary problem on Boltzmann]{Solutions to  a moving boundary problem on the Boltzmann equation}

\author[R.-J. Duan]{Renjun Duan}
\address[R.-J. Duan]{Department of Mathematics, The Chinese University of Hong Kong, Hong Kong}
\email{rjduan@math.cuhk.edu.hk}

\author[Z. Zhang]{Zhu Zhang}
\address[Z. Zhang]{Department of Applied Mathematics, The Hong Kong Polytechnic University, Hong Kong}
\email{zhuama.zhang@polyu.edu.hk}

\begin{abstract}
Motivated by the numerical investigation by Aoki et al. \cite{AKFG}, we study a rarefied gas flow between two parallel infinite plates of the same temperature governed by the Boltzmann equation with diffuse reflection boundaries, where one plate is at rest and the other one oscillates in its normal direction periodically in time. For such boundary-value problem, we establish the existence of a time-periodic solution with the same period, provided that the amplitude of the oscillating boundary is suitably small.  The positivity of the solution is also proved basing on the study of its large-time asymptotic stability for the corresponding initial-boundary value problem. For the proof of existence, we develop uniform estimates on the approximate solutions in the time-periodic setting and make a bootstrap argument by reducing the coefficient of the extra penalty term from a large enough constant to zero. 

\vspace{3mm}
\begin{center}
\it Dedicated to Professor Zhouping Xin on the occasion of his 65th birthday
\end{center} 
\end{abstract}

\date{\today}
\maketitle

\tableofcontents

\thispagestyle{empty}

\section{Introduction}

\subsection{Problem}
We consider the motion with slab symmetry for a rarefied gas flow between two infinite plates parallel to each other in space. The left plate located at $x=0$ is stationary, while the right one with the position $X_w(t)$ and the normal velocity $V_w(t)=\dot{X}_w(t)$ oscillates near $x=1$  time-periodically with a period $T>0$ in its normal direction; see the recent work \cite{TA} and \cite{AKFG} by Aoki together with his collaborators. If the distance between two plates is comparable to the mean free path of gas particles, the continuum formulation is no longer valid to model such a situation and a kinetic description is necessary. In the kinetic setting, the problem can be reduced to the following moving boundary problem on the Boltzmann equation:
\begin{align}\label{1.1.1}
\pa_tF+v_1\pa_xF=\f1{{\rm Kn}}Q(F,F).
\end{align}
Here, the unknown $F=F(t,x,v)\geq 0$ stands for the density distribution function of gas particles with position $x\in \Omega(t):=(0,X_w(t))$ and velocity $v=(v_1,v_2,v_3)\in \R^3$ at time $t\in \R$. The constant ${\rm Kn}>0$ is the Knudsen number proportional to the mean free path of gas particles. The collision term $Q(F,F)$ acts on velocity variable $v$ only, and it takes the bilinear form:
\begin{align}
Q(F_1,F_2)&=\int_{\mathbb{R}^3}\int_{\mathbb{S}^2} B(v-u,\om)\left[F_1(u')F_2(v')
-F_1(u)F_2(v)\right]\,\dd\omega\dd u.\nonumber
\end{align}
In the above integral the post-collision velocity pair $(v',u')$ and the pre-collision velocity pair $(v,u)$ satisfy the relation
\begin{equation}
v'=v-[(v-u)\cdot\omega]\,\omega,\quad u'=u+[(v-u)\cdot\omega]\,\omega,\nonumber
\end{equation}
with $\om\in \S^2$, according to conservations of momentum and energy of two particles before and after the collision:
\begin{equation}
v+u=v'+u',\quad 
|v|^2+|u|^2=|v'|^2+|u'|^2.\nonumber
\end{equation}
The collision kernel $B(v-u,\om)$ models the intermolecular interaction. In this paper, we consider the hard sphere model, namely,
$B(v-u,\om)=|(v-u)\cdot \omega|$. 

To solve \eqref{1.1.1}, one has to supplement a boundary condition. We assume that on the surface of the plates, the gas particles undergo the diffuse reflection given by
\begin{equation}\label{1.1.2}\left\{
\begin{aligned}
&F(t,0,v)|_{v_1>0}=\sqrt{2\pi}\mu(v)\int_{v_1<0}F(t,0,v)|v_1|\,\dd v,\\
&F(t,X_{w}(t),v)|_{v_1<V_{w}(t)}=\sqrt{2\pi}\mu_{w}(t,v)\int_{v_1>V_w(t)}F(t,X_{w}(t),v)|v_1-V_{w}(t)|\,\dd v,
\end{aligned}\right.
\end{equation}
where two Maxwellians 
$$
\mu(v)=\frac{1}{(2\pi)^{\f32}}e^{-\f{|v|^2}{2}}, 
\quad \mu_w(t,v)=\frac{1}{(2\pi)^{\f32}}e^{-\f{|v_1-V_w(t)|^2+|v_2|^2+|v_3|^2}{2}}
$$ 
correspond to 
the boundary thermal equilibrium  on the left and right plates,  respectively. Notice that from \eqref{1.1.1} and \eqref{1.1.2}, one has the following conservation of mass:
$$
\f{\dd}{\dd t}\int_0^{X_w(t)}\int_{\mathbb{R}^3}F(t,x,v)\,\dd v\dd x=0,
$$
for any $t\geq 0$. Without loss of generality, we assume that for $t\geq 0$,
\begin{align}\label{mass}
\int_0^{X_w(t)}\int_{\mathbb{R}^3}F(t,x,v)\,\dd v\dd x\equiv 1.
\end{align}

The moving boundary problem is fundamental in kinetic theory and can be applied to various physical contexts. For example, if  one regards the oscillating plate as a micro-mechanical structure, it can be used to model the gas vibrating in micro-electro-mechanical systems devices, see Desvillettes and Lorenzani \cite{DL}. If considering the oscillating plate as a sound source and the stationary plate as the receptor, it relates to the  propagation of nonlinear sound waves, cf.~Kalempa and Sharipov \cite{KS}. Due to its importance, such a problem has received much attention and is treated by lots of numerical methods, for instance, the semi-Lagrangian method Russo and Filbet \cite{RF}.  We also would mention that an accurate numerical analysis on this problem has been investigated in Aoki et al \cite{AKFG} in the small Knudsen number regime (${\rm Kn}\ll1$). In the present paper, we are interested in a regime where ${\rm Kn}=O(1)$. In this regime, a numerical analysis was carried out and a detailed description of solution behavior such as momentum and energy transfer was given in Tsuji and Aoki \cite{TA}. It is worthy to point out that the numerical results in \cite{AKFG} show that the flow field eventually approaches a time-periodic state in large time. However, there is still a lack of mathematical justification of the existence of such a time-periodic state. The main objective in the present paper is to study this problem with a focus on the existence of time-periodic states in the case when the right plate has a small time-periodic oscillation around a fixed position.

The time periodic problem is one important topic in both kinetic theory and gas dynamics. We recall that the existence and stability of time-periodic solutions for the Navier-Stokes equations in different settings were investigated, for instance, see Beir\~{a}o da Veiga \cite{B}, Feireisal et al \cite{FMPS}, Tsuda \cite{T}, and Valli and Zajaczkowski \cite{VZ}, and the references therein. For the Boltzmann equation with a time-periodic inhomogeneous source, the issue was studied in Ukai \cite{U} and Ukai and Yang \cite{UY}. When one takes into account the effect of a time-periodic external force in the whole space, the problem was partially solved in \cite{DUYZ}, subject to a restriction assumption that the spatial dimensions are not less than five, while so far it has remained open in the physical three-dimensional case. Recently, a result on this problem  was obtained for the Boltzmann equation in general 3D bounded domains with time-periodic boundary conditions \cite{DWZ}, where the boundedness of the domain inducing the fast enough time-decay plays a role. Thus, the time-periodic problem in an unbounded domain with boundaries is still very  challenging. Finally, we remark that recently there have been externsive studies on the steady solutions to Boltzmann equations and corresponding hydrodynamic limits, cf.~\cite{AEMN,CK,DLiu,DLY,EGKM,EGKM2,EGM1,EGMW}.

\subsection{Main result}
The main result of the present paper is stated as follows.

\begin{theorem}\label{thm1.1}
Set ${\rm Kn}=1$. Let
\begin{equation}
\label{def.vw}
w(v)=(1+|v|^2)^{\beta/2}e^{\f{q|v|^2}{4}}
\end{equation} 
be a velocity weight function with $\beta>3$ and $0\leq q<1$. Assume that $X_{w}(t)\in C^2(\R)$ is periodic in time with the period $T>0$. There are constants $\delta_0>0$ and $\hat{C}>0$ such that if 
$$
\delta:=\|X_w-1\|_{C^2}=\sup_{t\in \R} (|X_w(t)-1|+|X_w'(t)|+|X_w''(t)|)\leq \delta_0,
$$
then the moving boundary problem on the Boltzmann equation \eqref{1.1.1}, \eqref{1.1.2}, and \eqref{mass} admits a unique mild solution (see definition \ref{def1.1} later) 
\begin{equation}
\label{thm.adF}
F(t,x,v)=M(t,x,v)+\sqrt{M(t,x,v)}f(t,x,v)\geq 0,
\end{equation}
which is time periodic with the same period $T>0$ and satisfies
\begin{align}\label{1.1.3}
\sup_{t\in \mathbb{R},x\in\Omega(t),v\in\mathbb{R}^3}\left|wf(t,x,v)\right|\leq \hat{C}\delta.
\end{align}
Here the local Maxwellian in \eqref{thm.adF} is given as 
$$
M(t,x,v):=\f{1}{(2\pi)^{3/2}}e^{-\f{|v_1X_{w}(t)-xV_{w}(t)|^2+|v_2|^2+|v_3|^2}{2}}.
$$
\end{theorem}

In what follows, we briefly state the key procedure in the proof of the above result. Note that the problem \eqref{1.1.1}, \eqref{1.1.2} and \eqref{mass} is postulated under a moving frame. By making a change of variables \eqref{1.2.1}, it is more conveniently reformulated as a time-periodic problem in a fixed interval with an external time-periodic force. Such a driving force is produced by the effect of the right oscillating boundary. Thus, the transformation \eqref{1.2.1} is a basic point in the proof.

To prove the existence of time-periodic solutions, we develop new estimates compared to the steady problem. The first step is to establish some a priori $L^\infty$-estimates on the solutions. This is achieved by the $L^2$-$L^\infty$ interplay approach developed by Guo \cite{Guo2} and Esposito et al \cite{EGKM,EGKM2}, with an extra effort on treating the external force. A suitable smallness condition on the amplitude and frequency of the boundary oscillation induces the smallness of the such an external force, which is crucial for closing the $L^\infty$-estimates. 

The second step is to suitably design a sequence of approximate solutions, which is quite different with those in \cite{DWZ} and \cite{EGKM,EGKM2}. In fact, we should point out that in order to derive $L^\infty$ bounds of approximate solutions in terms of inhomogeneous data, it is necessary to first prove that the approximate solutions have the finite $L^\infty$ norm; see \cite{DW}. For instance, for a nonnegative quantity $A$ which may be infinite and for a positive finite quantity $B$, the estimate $A\leq \frac{1}{2}A+B$ does not imply $A\leq 2B$ in case when the possibility of $A=\infty$ is not excluded. Therefore one has to be careful in both constructing approximate solutions and obtaining their $L^\infty$ bounds.  One key point is to start from solving  the following problem with a penalty term in case when its magnitude $\la=\lambda_0>0$ is large enough:
\begin{align*}
\pa_tf^{n+1,\lambda} +v_1\pa_xf^{n+1,\lambda}+\sqrt{\mu}^{-1}G(t,x)\pa_{v_1}(\sqrt{\mu}f^{n+1,\lambda})+[\nu+\lambda]f^{n+1,\lambda}= Kf^{n,\lambda}+g.
\end{align*}
The appearance of $\lambda>0$ is to ensure the 
total mass condition \eqref{mass}  and the largeness of $\lambda $ at the present step is crucial for obtaining the uniform-in-$n$ estimates on $L^2$ norm of $f^{n,\lambda}$. After taking the limit $n\rightarrow \infty$, by the bootstrap argument in our recent work  \cite{DHWZ,DWZ} we are able to to further construct solutions by reducing $\lambda=\lambda_0$ to $\lambda=0$. In the end, we establish the non-negativity of the time-periodic solution by showing that it is exponentially stable under the moving frame.


The rest of this paper is organized as follows. In Section \ref{sec2}, we summarize for the later use some basic facts including the  reformulation of the original problem and some related estimates on the collision operator. Section \ref{sec3} is the main part devoted to proving the existence of time-periodic solutions. In Section \ref{sec4} we establish the exponential asymptotic stability of the time-periodic solution of the reformulated problem which implies the non-negativity of the time-periodic solution.

\section{Preliminary}\label{sec2}

\subsection{Reformulation}
To remove the difficulty due to the time-dependent spatial domain, we introduce the following new coordinates:
\begin{align}\label{1.2.1}
\b{t}:=\int_0^tX_w^{-2}(\tau)\dd \tau,\  
\b{x}:=\f{x}{X_w(t)},\ 
\b{v}_1:= v_1X_w(t)-xV_w(t), \ 
\b{v}_i:=v_i,\ i=2,3.
\end{align}
Then \eqref{1.1.1}, \eqref{1.1.2}, {\eqref{mass}} can be rewritten as:
\begin{align}\label{1.2.2}
\pa_{\b{t}}F+\b{v}_1\pa_{\b{x}}F+G(\b{t},\b{x})\pa_{\b{v}_1}F=Q(F,F),\  \b{t}\in \mathbb{R}, \ \b{x}\in {\Omega}:= (0,1),\ \b{v}\in \mathbb{R}^3,
\end{align}
with the boundary conditions
\begin{equation}\label{1.2.3}\left\{
\begin{aligned}
&F(\b{t},0,\b{v})|_{\b{v}_1>0}=\sqrt{2\pi}\b{\mu}_{w}(\b{t},\b{v})\int_{\b{v}_1<0}F(\b{t},0,\b{v})|\b{v}_1|\,\dd \b{v},\\
&F(\b{t},1,\b{v})|_{\b{v}_1<0}=\sqrt{2\pi}\b{\mu}_{w}(\b{t},\b{v})\int_{\b{v}_1>0}F(\b{t},1,\b{v})|\b{v}_1|\,\dd \b{v},
\end{aligned}\right.
\end{equation}
where  we have denoted
$$
G(\b{t},\b{x})=-\b{x}X^3_w\dot{V}_w[t(\b{t})],\quad 
\b{\mu}_w(\b{t},\b{v})=\f{1}{(2\pi)^{3/2}X_w^2}e^{-\f{\left|\f{\b{v}_1}{X_w}\right|^2+|\b{v}_2|^2+|\b{v}_3|^2}{2}},
$$
and the total-mass condition holds true:
\begin{align}\label{mass1}
\int_0^1\int_{\mathbb{R}^3}F(\b{t},\b{x},\b{v})\dd \b{x}\dd \b{v}\equiv 1.
\end{align}
Note that $G(\b{t},\b{x})$, $X_w(t(\b{t}))$ and $V_{w}(t(\b{t}))$ are all time periodic functions with period $\b{T}=\int_0^TX^{-2}_{w}(t)\dd t.$ For convenience, we call the BVP \eqref{1.1.1}, \eqref{1.1.2}, \eqref{mass} as {\it Problem 1} and the BVP \eqref{1.2.2}, \eqref{1.2.3}, \eqref{mass1} as {\it Problem 2}. The equivalence between {\it  Problem 1} and {\it Problem 2} will be given in Lemma \ref{lm2.1.1}.

\subsection{Characteristics}
For any $(\b{t},\b{x},\b{v})\in\mathbb{R}\times (0,1)\times \mathbb{R}^3$, we define 1-d characteristics $[\b{X}(\b{s};\b{t},\b{x},\b{v}),\b{V}(\b{s};\b{t},\b{x},\b{v})]$ for {\it Problem 2} as the solution of the following ODEs
\begin{equation}\label{2.1.1}
\left\{
\begin{aligned}
&\f{\dd \b{X}(\b{s};\b{t},\b{x},\b{v})}{\dd \b{s}}=\b{V}_1(\b{s};\b{t},\b{x},\b{v}),\\
&\f{\dd \b{V}_1(\b{s};\b{t},\b{x},\b{v})}{\dd s}=G(\b{s},\b{X}(\b{s};\b{t},\b{x},\b{v})),\\
&\f{\dd \b{V}_2(\b{s};\b{t},\b{x},\b{v})}{\dd s}=\f{\dd \b{V}_3(\b{s};\b{t},\b{x},\b{v})}{\dd s}=0,\\
&[\b{X}(\b{s};\b{t},\b{x},\b{v}),V(\b{s},\b{t},\b{x},\b{v})]_{\b{s}=\b{t}}=[\b{x},\b{v}].
\end{aligned}
\right.
\end{equation}
We define the {\it backward exit time} $\b{t}_{\mathbf{b}}(\b{t},\b{x},\b{v})\geq 0$ to be the last moment at which the
back-time characteristics $\b{X}(\b{s};\b{t},\b{x},\b{v})$ remains in $(0,1)$, that is,
\begin{equation}
\b{t}_{\mathbf{b}}(\b{t},\b{x},\b{v})=\sup \{\b{s}\geq 0:\b{X}(\b{\tau};\b{t},\b{x},\b{v})\in(0,1)\text{ for any $\b{t}-\b{s}<\b{\tau}<\b{t}$ }\}.\nonumber
\end{equation}
We also define
\begin{equation}
\b{x}_\mathbf{b}(\b{t},\b{x},\b{v})=\b{X}(\b{t}-\b{t}_{\mathbf{b}};\b{t},\b{x},\b{v})\in \{0,1\} ,\quad \b{v}_\mathbf{b}(\b{t},\b{x},\b{v})=\b{V}(\b{t}-\b{t}_{\mathbf{b}};\b{t},\b{x},\b{v}).\nonumber
\end{equation}
For {\it Problem 1}, we can also define the characteristics $$[X(s;t,x,v),V(s;t,x,v)]=[x-(t-s)v_1,v]$$ for $(t,x,v)\in \mathbb{R}\times \Omega(t)\times \mathbb{R}^3$, which is the solution to the ODEs
$$
\left\{
\begin{aligned}
&\f{\dd X(s;t,x,v)}{\dd s}=V_1(s;t,x,v),\\
&\f{\dd V_1(s;t,x,v)}{\dd s}=\f{\dd V_2(s;t,x,v)}{\dd s}=\f{\dd V_3(s;t,x,v)}{\dd s}=0,\\
&[X(s;t,x,v),V(s,t,x,v)]_{s=t}=[x,v].
\end{aligned}
\right.
$$
Similarly, we also define
$$t_{\mathbf{b}}(t,x,v)=\sup \{s\geq 0:X(\tau;t,x,v)\in\Omega(\tau)\text{ for any $t-s<\tau<t$ }\},
$$
and
$$x_\mathbf{b}(t,x,v)=X(t-t_{\mathbf{b}};t,x,v).
$$
\begin{definition}\label{def1.1}
We say that $F(t,x,v)$ is a mild solution to Problem 1 if the following things are satisfied:

(1) for any $(t,x,v)\in \mathbb{R}\times \Omega(t)\times \mathbb{R}^3$ with $t_{\mathbf{b}}(t,x,v)>0$ and any $t-{t}_{\mathbf{b}}(t,x,v)<{s}< {t}$, 
$$
F({s}):= F({s},{X}({s};{t},{x},{v}),{V}(s;{t},{x},{v}))
$$ 
is differentiable with respect to ${s}$ and
$$
\f{\dd F({s})}{\dd{s}}=Q(F,F)(s).
$$

(2) $F({t},{x},{v})$ satisfies the boundary condition \eqref{1.1.2}  and the total-mass condition \eqref{mass}.
\end{definition}

Similarly, by using the characteristics \eqref{2.1.1}, we can define the mild solution to {\it Problem 2} as follows:
\begin{definition}
We say that $F(\b{t},\b{x},\b{v})$ is a mild solution to Problem 2 if the following things are satisfied:

(1) for any $(\b{t},\b{x},\b{v})\in \mathbb{R}\times (0,1)\times \mathbb{R}^3$ with $\b{t}_{\mathbf{b}}(\b{t},\b{x},\b{v})>0$ and any $\b{t}-\b{t}_{\mathbf{b}}(\b{t},\b{x},\b{v})<\b{s}< \b{t}$, $$F(\b{s}):= F(\b{s},\b{X}(\b{s};\b{t},\b{x},\b{v}),\b{V}(\b{s};\b{t},\b{x},\b{v}))$$ is differentiable with respect to $\b{s}$ and
$$\f{\dd F(\b{s})}{\dd \b{s}}=Q(F,F)(\b{s}).
$$

(2) $F(\b{t},\b{x},\b{v})$ satisfies the boundary condition \eqref{1.2.3} and the total-mass condition \eqref{mass1}.
\end{definition}

\begin{lemma}\label{lm2.1.1}(Equivalence between Problem 1 and Problem 2)
If $\b{F}(\b{t},\b{x},\b{v})$ is a mild solution to Problem 2, then $F(t,x,v):= \b{F}(\varphi(t,x,v))$ is a mild solution to Problem 1 and vice versa. Here the mapping
$$
\begin{aligned}
\varphi:\mathbb{R}\times \Omega(t)\times \mathbb{R}^3&\mapsto \mathbb{R}\times(0,1)\times \mathbb{R}^3\\
(t,x,v)&\mapsto\varphi(t,x,v)=(\b{t},\b{x},\b{v})
\end{aligned}
$$
is defined in \eqref{1.2.1}. Moreover, if $\b{F}(\b{t},\b{x},\b{v})$ is time-periodic with period $\b{T}>0$, then $F(t,x,v)$ is also time-periodic with period $T$, where $T$ is determined by the relation
$$
\b{T}=\int_0^TX_{w}^{-2}(t)\dd t.
$$
\end{lemma}

\begin{proof}
We first claim that
\begin{align}\label{2.1.2}
(\b{s},\b{X}(\b{s};\b{t},\b{x},\b{v}),\b{V}(\b{s};\b{t},\b{x},\b{v}))
&=\varphi(s,X(s;t,x,v),V(s;t,x,v)).
\end{align}
In fact, a direct computation shows that
\begin{align}
\f{\dd \overline{X(s;t,x,v)}}{\dd \b{s}}=\f{\dd\overline{X(s;t,x,v)}}{\dd s}\cdot\f{\dd s}{\dd \b{s}}=\overline{V_1(s;t,x,v)},\nonumber
\end{align}
and
\begin{align}
\f{\dd \overline{V_1(s;t,x,v)}}{\dd \b{s}}=\f{\dd\overline{V_1(s;t,x,v)}}{\dd s}\cdot\f{\dd s}{\dd \b{s}}=-\overline{X(s;t,x,v)}X_w^3(s)\dot{V}_w(s)=G(\b{s},\overline{X(s;t,x,v)}).\nonumber
\end{align}
Hence $[\overline{X(s;t,x,v)},\overline{V(s;t,x,v)}]$ solves the characteristic ODEs \eqref{2.1.1}. Since it is deterministic, $$[\overline{X(s;t,x,v)},\overline{V(s;t,x,v)}]\equiv[\b{X}(\b{s};\b{t},\b{x},\b{v}),\b{V}(\b{s};\b{t},\b{x},\b{v})].$$ This proves \eqref{2.1.2}.

Let $\b{F}(\b{t},\b{x},\b{v})$ be the mild solution to Problem 2. For any $(t,x,v)\in \mathbb{R}\times \Omega(t)\times \mathbb{R}^3,$
we have 
$$
X(\tau;t,x,v)\in \Omega(t)\text{ iff }\overline{X(\tau;t,x,v)}\in (0,1),\text{ iff }\b{X}(\b{s};\b{t},\b{x},\b{v})\in (0,1),
$$ 
which implies that
$$
\overline{t_{\mathbf{b}}(t,x,v)}=\b{t}_{\mathbf{b}}(\varphi(t,x,v)).
$$ 
Then for any $(t,x,v)\in \mathbb{R}\times \Omega(t)\times \mathbb{R}^3,$
 and any $t_{\mathbf{b}}(t,x,v)<s<t$, we have $\b{t}_{\mathbf{b}}(\varphi(t,x,v))<\b{s}<\b{t}$. We hence compute
$$
\f{\dd F(s)}{\dd s}=\f{\dd \b{F}}{\dd \b{s}}\cdot\f{\dd{\b{s}}}{\dd s}=Q(\b{F},\b{F})(\b{s})\cdot X^{-2}_w(s)=Q(F,F)(s).
$$
The verification of boundary conditions, total-mass condition \eqref{mass} and periodicity are straightforward. The proof of Lemma \ref{lm2.1.1} is complete.
\end{proof}

To the end we are devoted for constructing the time-periodic solution to {\it Problem 2}. For simplicity of notation, we drop `bar' in the sequel. The following lemma shows the time-periodicity of $t_{\mathbf{b}}$, $x_{\mathbf{b}}$ and $v_{\mathbf{b}}$, which is crucial for the construction of time-periodic solutions to \eqref{1.2.2}, \eqref{1.2.3}, \eqref{mass1} later on.

\begin{lemma}\label{lm2.1.2}
Let $G(t,x)$ be a time-periodic function with period $T>0$ and 
$$
[X(s;t,x,v), V(s;t,x,v)]
$$ 
be the solution to the characteristic ODEs \eqref{2.1.1}. Then we have
\begin{align}\label{2.1.3}
X(s+T;t+T,x,v)=X(s;t,x,v), \quad V(s+T;t+T,x,v)=V(s;t,x,v).
\end{align}
Moreover, $t_{\mathbf{b}}(t,x,v)$, $x_{\mathbf{b}}(t,x,v)$ and $v_{\mathbf{b}}(t,x,v)$ are all time periodic functions with the same period $T.$
\end{lemma}

\begin{proof}
Define $[X_T(s),V_{T}(s)]:=[X(s+T;t+T,x,v),V(s+T;t+T,x,v)]$. Then $$[X_T(s),V_T(s)]_{s=t}=[x,v].$$
By using time-periodicity of $G(t,x)$, it is straightforward to verify that
$$\f{\dd X_T(s)}{\dd s}=V_T(s),\quad \f{\dd V_T(s)}{\dd s}=G(s+T,X_T(s))=G(s,X_T(s)).
$$
That means that $[X_T(s),V_T(s)]$ satisfies the characteristic ODEs \eqref{2.1.1}. Therefore, \eqref{2.1.3} follows. We also have
\begin{align}
t_{\mathbf{b}}(t+T,x,v)&=\sup \{s\geq 0:X(\tau;{t+T},{x},{v})\in(0,1)\text{ for any ${t}+T-{s}<{\tau}<{t}+T$}\}\nonumber\\
&=\sup \{s\geq 0:X(\tau-T;{t},{x},{v})\in(0,1)\text{ for any ${t}+T-{s}<{\tau}<{t}+T$}\}\nonumber\\
&=\sup \{s\geq 0:X(\tau;{t},{x},{v})\in(0,1)\text{ for any ${t}-{s}<{\tau}<{t}$}\}\nonumber\\
&=t_{\mathbf{b}}(t,x,v),\nonumber\\
x_{\mathbf{b}}(t+T,x,v)&=X(t+T-t_\mathbf{b};t+T,x,v)=X(t-t_{\mathbf{b}};t,x,v)=x_{\mathbf{b}}(t,x,v),\nonumber\\ v_{\mathbf{b}}(t+T,x,v)&=V(t+T-t_\mathbf{b};t+T,x,v)=V(t-t_{\mathbf{b}};t,x,v)=v_{\mathbf{b}}(t,x,v).\nonumber
\end{align}
This completes the proof of Lemma \ref{lm2.1.2}.
\end{proof}
\subsection{Estimates on collision operators}
Fix a global Maxwellian $\mu(v)$. Recall the linearized collision operator
\begin{align}
Lf:=-\f{1}{\sqrt{\mu}}[Q(\mu,\sqrt{\mu}f)+Q(\sqrt{\mu}f,\mu)]:=\nu(v)f-Kf,\nonumber
\end{align}
where
$$\nu(v)=\int_{\mathbb{R}^3}\int_{\mathbb{S}^2}B(v-u,\omega)\mu(u)\,\dd\omega\dd u\sim 1+|v|,
$$
and $K=K_1-K_2$  is defined by
\begin{align}
(K_1f)(v)&=\int_{\mathbb{R}^3}\int_{\mathbb{S}^2}B(v-u,\omega)\sqrt{\mu(v)\mu(u)}f(u)\,\dd\omega\dd u,\nonumber\\
(K_2f)(v)&=\int_{\mathbb{R}^3}\int_{\mathbb{S}^2}B(v-u,\omega)\sqrt{\mu(u)\mu(u')}f(v')\,\dd\omega\dd u\notag\\
&\qquad+\int_{\mathbb{R}^3}\int_{\mathbb{S}^2}B(v-u,\omega)\sqrt{\mu(u)\mu(v')}f(u')\,\dd\omega\dd u.\nonumber
\end{align}
The nonlinear term $\Ga(f,g)=\Ga_+(f,g)-\Ga_-(f,g)$ is defined as
\begin{align*}
\Gamma_+(f,g)
=\f1{\sqrt{\mu}}Q_+(\sqrt{\mu}f,\sqrt{\mu}g),\quad \Ga_-(f,g)=\f1{\sqrt{\mu}}Q_-(\sqrt{\mu}f,\sqrt{\mu}g).
\end{align*}

\begin{lemma}[\cite{Guo2}]
The operator $L$ is self-adjoint and non-negative on $L^2_v$. The kernel of $L$ is a five-dimensional space spanned by the following bases: 
$$
\chi_0:=\sqrt{\mu},\quad \chi_i:=v_i\sqrt{\mu},\quad i=1,2,3,\quad \chi_4:=\f{|v|^2-3}{\sqrt{6}}\sqrt{\mu}.
$$
Define the projection $P$:
\begin{align}\label{P}
Pf:=\sum_{i=0}^4\langle f,\chi_i\rangle\chi_i.
\end{align}
Then there exists a constant $c_0>0$ such that
\begin{align}\label{c}
\langle Lf,f\rangle\geq c_0|\nu^{1/2}(I-P)f|_{L^{2}(\mathbb{R}^3)}^2.
\end{align}
\end{lemma}

\begin{lemma}[\cite{Gra,Guo2}]
It holds that $K$ is an integral operator given by
$$
Kf:=\int_{\mathbb{R}^3}k(v,u)\,\dd u,
$$
where
\begin{align}\label{k1}
	|k(v,u)|\leq C\left\{|v-u| +|v-u|^{-1} \right\}e^{-\f{|v-u|^2}{8}}e^{-\f{||v|^2-|u|^2|^2}{8|v-u|^2}},
\end{align}
for any $v, u\in \R^3$ with $v\neq u$.
Moreover, let $\beta\geq 0$ and $0\leq q\leq 1$. There is $C_{q,\beta}>0$ depending only on $\beta$ and $q$ such that
\begin{align}\label{k2}
\int_{\mathbb{R}^3} \left|k(v,u)\f{e^{\f{q|v|^2}{4}}}{e^{\f{q|u|^2}{4}}}\right|(1+|u|)^{-\beta}\,\dd\eta\leq C_{q,\beta}(1+|v|)^{-1-\beta},
\end{align}
for any $v\in \R^3$.
\end{lemma}

\begin{lemma} [\cite{Guo2}]
The nonlinear term $\Gamma$ satisfies that for $\beta\geq 0$ and $0\leq q<1$,
\begin{align}\label{g}
\|\nu^{-1}\Gamma(f,g)\|_{L^{\infty}}\leq C\|wf\|_{L^\infty}\|wg\|_{L^\infty}.
\end{align}
\end{lemma}

\subsection{An iteration lemma}
The following iteration lemma will be crucially used later.

\begin{lemma}[\cite{DHWZ}]
Consider a sequence $\{a_i\}_{i=0}^\infty $  with $a_i\geq0$ for $i=0,1,\cdots$. For any fixed $k\in\mathbb{N}_+$, we denote 
$$
A_i^k=\max\{a_i, a_{i+1},\cdots, a_{i+k}\}.
$$

\noindent{(1)} Assume $D\geq0$.  If $a_{i+1+k}\leq \f18 A_i^{k}+D$ for $i=0,1,\cdots$, then it holds that
\begin{equation}\label{A.1}
A_i^k\leq \left(\f18\right)^{\left[\frac{i}{k+1}\right]}\cdot\max\{A_0^k, \ A_1^k, \cdots, \ A_k^k \}+\f{8+k}{7} D,\quad\mbox{for}\quad i\geq k+1.
\end{equation}

\noindent{(2)} 
Let $0\leq \eta<1$ with $\eta^{k+1}\geq\frac14$.  If $a_{i+1+k}\leq \f18 A_i^{k}+C_k \cdot \eta^{i+k+1}$ for $i=0,1,\cdots$, then it holds that
\begin{align}\label{A.1-1}
A_i^k\leq \left(\f18\right)^{\left[\frac{i}{k+1}\right]}\cdot\max\{A_0^k, \ A_1^k, \cdots, \ A_k^k \}+2C_k\f{8+k}{7} \eta^{i+k},\quad\mbox{for}\quad i\geq k+1.
\end{align}
\end{lemma}

\section{Existence of time-periodic solution}\label{sec3}
In this section, we will construct the time-periodic solution to the reformulated problem \eqref{1.2.2}, \eqref{1.2.3}. We first list some notations and functional spaces for later use. For $x\in\pa\Omega=\{0,1\}$, we define the outward normal vector
\begin{equation}\nonumber
n(x)=\left\{
\begin{aligned}
(-1,0,0),\quad &x=0,\\
(1,0,0),\quad &x=1.
\end{aligned}
\right.
\end{equation}
 Denote the phase boundary $\gamma:=\{0,1\}\times \mathbb{R}^3=\gamma_+\cup\gamma_0\cup\gamma_-$, where
\begin{align}
\gamma_\pm=(\{0\}\times\{v_1\lessgtr0\})\cup(\{1\}\times\{v_1\gtrless0\}),\quad \gamma_0=\{0,1\}\times\{v_1=0\}.\nonumber
\end{align}
Define the Hilbert space $L^2(\gamma_\pm)$, equipped with the inner product
\begin{align}
\langle f,g\rangle_{\gamma_\pm}:=\mp\int_{v_1\lessgtr0}f(0,v)g(0,v) v_1\dd v\pm\int_{v_1\gtrless 0}f(1,v)g(1,v)v_1\dd v.\nonumber
\end{align}
We denote $|\cdot|_{L^2_\pm}$ as the norm induced by the inner product $\langle\cdot,\cdot\rangle_{\g_\pm}$.
For any $f\in L^2(\gamma_+)$, we define $P_{\gamma}f$ as
$$P_{\gamma}f(0,v)=\sqrt{2\pi\mu}\int_{v_1<0}f(0,v)\sqrt{\mu}|v_1|\dd v,\quad P_{\gamma}f(1,v)=\sqrt{2\pi\mu}\int_{v_1>0}f(1,v)\sqrt{\mu}v_1\dd v.$$
Note that $P_\gamma$ can be also viewed as an orthogonal projection operator on $L^2(\gamma_+)$.
We denote $\langle\cdot,\cdot\rangle$ as the standard $L^2((0,1)\times\mathbb{R}^3_v)$-inner product and $\|\cdot\|_{L^2}$ as its norm. $\|\cdot\|_{L^\infty}$ denotes the $L^\infty((0,1)\times\mathbb{R}^3_v)$-norm. Moreover, we denote 
$$
\|f\|_{L^{2}(0,T;L^2)}=\left(\int_0^T\|f(s)\|_{L^2}^2\dd s\right)^{1/2}\quad \text{and}\quad \|f\|_{L^{\infty}_{t,x,v}}:=\sup_{t\in\mathbb{R}}\|f(t)\|_{L^{\infty}}.
$$ 
For the phase boundary integration, we denote $d\sigma=|n(x)\cdot v|\mu(v)\dd v.$ We also denote 
$$
|f|_{L^{\infty}_{\pm}}:=\sup_{(x,v)\in\g_{\pm}}|f(x,v)|,
\quad\text{and}\quad|f|_{L^{\infty}_{t}L^{\infty}_{\pm}}:=\sup_{t\in\mathbb{R}}|f(t)|_{L^{\infty}_{\pm}}.
$$
Recall that the velocity weight $w$ is given in \eqref{def.vw}.

\begin{theorem}\label{thm1.2}
Let $\beta>3$ and $0\leq q<1$. There are $\delta_1>0$ and $C>0$ such that if $X_w\in C^2$ is time periodic with period $T$ and satisfies
\begin{align}
\|X_{w}-1\|_{C^2}=\delta\leq \delta_1,\nonumber
\end{align}
then the problem \eqref{1.2.2}, \eqref{1.2.3} admits a unique time periodic solution with the same period $T$ taking the form  
$$F^{\text{per}}(t,x,v)=\mu(v)+\sqrt{\mu}f^{\text{per}}(t,x,v),$$ which satisfies $\int_0^1\int_{\mathbb{R}^3}f^{\text{per}}(t,x,v)\sqrt{\mu(v)}\dd v\dd x\equiv0$ and
\begin{align}\label{3.0.5}
\|w f^{\text{per}}(t)\|_{L^\infty_{t,x,v}}+|w f^{\text{per}}|_{L^{\infty}_tL^\infty_{\pm}}\leq C\d.
\end{align}
\end{theorem}

The proof of Theorem \ref{thm1.2} heavily relies on the solvability of the following linear problem:
\begin{equation}\label{3.0.1}\left\{
\begin{aligned}
&\pa_tf+v_1\pa_xf+G(t,x)\pa_{v_1}\left(\sqrt{\mu}f\right)\f{1}{\sqrt{\mu}}+Lf=g, \ (t,x,v)\in \R\times (0,1)\times \R^3,\\
&f|_{\gamma_-}=P_\gamma f+r.
\end{aligned}
\right.
\end{equation}
Here the force $G(t,x)$ and inhomogeneous sources $g$ and $r$ are all time-periodic functions with period $T>0.$
\begin{proposition}\label{prop3.1}
Let $\beta>3$ and $0\leq q<1$. Assume that $G$, $g$ and $r$ are time-periodic functions with period $T>0$, and satisfy the following zero-mass condition
\begin{align}\label{3.0.2}
\int_0^1\int_{\mathbb{R}^3}g(t,x,v)\sqrt{\mu(v)}\dd v\dd x=\langle r,\sqrt{\mu}\rangle_{L^{2}_-}=0,
\end{align}
for all $t\in \mathbb{R}$, and $L^{\infty}$ bounds
$$\|\nu^{-1}wg\|_{L^\infty_{t,x,v}}+|wr|_{L^{\infty}_tL^\infty_-}<\infty.$$  Then if $\|X_w-1\|_{C^2}$ is sufficiently small, there exists a unique time-periodic solution $f=f(t,x,v)$ with the same period, to the linearized Boltzmann equation \eqref{3.0.1}, such that 
$$
\int_0^1\int_{\mathbb{R}^3} f(t,x,v) \sqrt{\mu} dvdx=0 
$$ 
for all $t\in \mathbb{R}$, and
\begin{equation}\label{3.0.3}
\|wf\|_{L^\infty_{t,x,v}} +|wf|_{L^\infty_tL^\infty_\pm} \leq C |wr|_{L^{\infty}_{t}L^\infty_-}+C\|\nu^{-1}wg\|_{L^\infty_{t,x,v}}.
\end{equation}
\end{proposition}
\subsection{$L^{\infty}$-estimate}
Denote $h(t,x,v):= wf(t,x,v)$. Then the equation for $h$ reads as
\begin{align}
\begin{cases}
\pa_th+v_1 \pa_x h+G(t,x)\cdot\pa_{v_1}h+\tilde{\nu}(t,x,v)h=K_{w} h+wg,\\[2mm]
h|_{\g_-}=\f{1}{\tilde{w}(v)} \int_{n(x)\cdot v'>0} h(t,x,v') \tilde{w}(v') \dd\sigma+wr,\nonumber
\end{cases}
\end{align}
where we have denoted 
$$
\tilde{\nu}=\nu(v)-\f{G(t,x)v_1}{2}-\f{G(t,x)\pa_{v_1}w}{w}, \quad \tilde{w}(v)\equiv \f{1}{w(v)\sqrt{\mu(v)}},\quad
K_wh=wK(\f{h}{w}),
$$
and 
$
\dd \si=\dd \si(v')=\sqrt{2\pi}\mu(v')|n(x)\cdot v'|\,\dd v'.
$
Note that $\dd \si$ is a probability measure on $\{n(x)\cdot v'>0\}$.  In what follows we are devoted to establishing the uniform $L^{\infty}$-estimate on the solution to the following time-periodic problem:
\begin{equation}\label{3.1.2}
\begin{cases}
\pa_th^{i+1}+v_1\pa_x h^{i+1}+G(t,x)\cdot\pa_{v_1}h^{i+1}+(\lambda+\tilde{\nu}(t,x,v)) h^{i+1}=K_w h^i +wg,\\[2mm]
h^{i+1}|_{\g_-}=\f{1}{\tilde{w}(v)} \int_{n(x)\cdot v'>0} h^i(t,x,v') \tilde{w}(v')\, \dd\sigma+wr,
\end{cases}
\end{equation}
for $i=1,2,3,\cdots$ and $h^0=h^0(t,x,v)$ is given. Here $\lambda$ is a positive constant, and $g(t,x,v)$ and $r(t,x,v)$ are both time-periodic functions with period $T>0$. Before doing that, we need some preparations. Let $t\in \mathbb{R},$ $(x,v)\in \big[(0,1)\times \mathbb{R}^3\big]\cup\gamma_+$ and
$
(t_{0},x_{0},v_{0})=(t,x,v)$. For $v_{k+1}\in {\mathcal{V}}_{k+1}:=\{v_{k+1}\cdot n(x_k)<0 \}$, the back-time cycle is defined as
\begin{equation}
\left\{\begin{aligned}
X_{cl}(s;t,x,v)&=\sum_{k}\Fi_{[t_{k+1},t_{k})}(s)X(s;t_k,x_k,v_k),\\[1.5mm]
V_{cl}(s;t,x,v)&=\sum_{k}\Fi_{[t_{k+1},t_{k})}(s)V(s;t_k,x_k,v_k),\nonumber
\end{aligned}\right.
\end{equation}
with
\begin{equation}
({t}_{k+1},{x}_{k+1},v_{k+1})
=({t}_{k}-{t}_{\mathbf{b}}(t_k,{x}_{k},v_{k}), {x}_{\mathbf{b}}(t_k,{x}_{k},v_{k}),v_{k+1}).\nonumber
\end{equation}

The following lemma gives the mild formulation of $h^{i+1}$ in \eqref{3.1.2}, and its proof is omitted for brevity, cf.~\cite{Guo2}.

\begin{lemma}\label{lm3.1}
Let $\lambda>0$ and integer $k\geq 1 $. For any $t\in [0,T]$, almost every $(x,v)\in\big[(0,1)\times \mathbb{R}^3\big]\cup\gamma_+$, any $s\leq t$ and for any $i\geq k-1$, we have
\begin{align}\label{3.1.3}
h^{i+1}=\sum_{i=1,2,3}J_i+\Fi_{\{t_1>s\}}\sum_{i=4}^{11}J_i,
\end{align}
where we have denoted
$$
\begin{aligned}
	&J_1=\Fi_{\{t_1\leq  s\}}e^{-\int_s^t(\tilde{\nu}(\tau')+\lambda)\dd \tau'} h^{i+1}(s,X_{cl}(s),V_{cl}(s)),\\
	&J_2+J_3=\int_{\max\{{t}_1,s\}}^t e^{-\int_{\tau}^t(\tilde{\nu}(\tau')+\lambda)\dd \tau'}\Big[K_wh^{i} +wg\Big](\tau,X_{cl}(\tau),V_{cl}(\tau))\dd \tau,\\
	&J_4=e^{-\int_{t_1}^t({\tilde{\nu}}(\tau')+\lambda)\dd\tau} w r(t_1,x_1,V_{cl}(t_1)),\\
	& J_5=\frac{e^{-\int_{t_1}^t(\tilde{\nu}(\tau')+\lambda)\dd\tau'}}{\tilde{w}(V_{cl}(t_1))} \int_{\Pi _{j=1}^{k-1}\mathcal{V}_{j}}
	\sum_{l=1}^{k-2} \Fi_{\{t_{l+1}>s\}}
 wr(t_{l+1},x_{l+1},V_{cl}(t_{l+1}))\dd \Sigma_{l}({t}_{l+1}),\\
 &J_6=\frac{e^{-\int_{t_1}^t(\tilde{\nu}(\tau')+\lambda)\dd \tau'}}{\tilde{w}(V_{cl}(t_1))} \int_{\Pi _{j=1}^{k-1}{\mathcal{V}}_{j}} \sum_{l=1}^{k-1} \Fi_{\{{t}_{l+1}\leq s<{t}_l\}} h^{i+1-l}(s,X_{cl}(s),V_{cl}(s)) \dd\Sigma_{l}(s),
    \end{aligned}
$$
$$
\begin{aligned}
	&J_7+J_8=\frac{e^{-\int_{t_1}^t(\tilde{\nu}(\tau')+\lambda)\dd\tau'}}{\tilde{w}(V_{cl}(t_1))} \int_{\Pi_{j=1}^{k-1}{\mathcal{V}}_{j}} \sum_{l=1}^{k-1}\int_{s}^{{t}_l} \Fi_{\{{t}_{l+1}\leq s<t_{l}\}}[K_wh^{i-l}+wg](\tau,X_{cl}(\tau),V_{cl}(\tau)) \dd\Sigma_l(\tau)\dd\tau,\\
	&J_9+J_{10}=\frac{e^{-\int_{t_1}^t(\tilde{\nu}(\tau')+\lambda)\dd\tau'}}{\tilde{w}(V_{cl}(t_1))} \int_{\Pi_{j=1}^{k-1}{\mathcal{V}}_{j}} \sum_{l=1}^{k-1}\int_{{t}_{l+1}}^{{t}_l} \Fi_{\{{t}_{l+1}>s\}}[K_wh^{i-l}+wg](\tau,X_{cl}(\tau),V_{cl}(\tau)) \dd\Sigma_l(\tau)\dd\tau,\\
	&J_{11}=\frac{e^{-\int_{t_1}^t(\tilde{\nu}(\tau')+\lambda)\dd\tau'}}{\tilde{w}(V_{cl}(t_1))} \int_{\Pi _{j=1}^{k-1}{\mathcal{V}}_{j}}  \Fi_{\{{t}_{k}>s\}} h^{i+2-k}(t_k,{x}_k,V_{cl}(t_k)) \dd\Sigma_{k-1}({t}_k),
    \end{aligned}
$$
and
\begin{align}
\dd\Sigma_l(\tau) = \big\{\Pi_{j=l+1}^{k-1}\dd{\sigma}_j\big\}\cdot \big\{\tilde{w}(v_l) e^{-\int_{\tau}^{t_l}(\tilde{\nu}(\tau')+\lambda)\dd\tau'} \dd{\sigma}_l\big\}\cdot \big\{\Pi_{j=1}^{l-1} e^{-\int_{t_{j+1}}^{t_{j}}(\tilde{\nu}(\tau')+\lambda)\dd\tau'} \f{\tilde{w}(v_j)}{\tilde{w}(V_{cl}(t_{j+1}))}\dd{\sigma}_j\big\}.\nonumber
\end{align}
\end{lemma}

\begin{lemma}\label{lm3.2}
For $T_0\gg1$ be sufficiently large, there exists a positive constant $\delta_2=\delta_2(T_0)$, such that if $\|G\|_{L^{\infty}_{t,x}}\leq \delta_2$, then there exist constants $C_1$, $C_2>0$, independent of $T_0$, such that for $k=C_1T_0^{5/4}$, any $s\in\mathbb{R}$ and any $(t,x,v)\in\{[s,s+T_0]\times(0,1)\times\mathbb{R}^3\}\cup\{[s,s+T_0]\cup\g_+\}$, it holds that
\begin{align}\label{3.1.3-1}
\int_{\Pi _{j=1}^{k-1}{\mathcal{V}}_{j}} \Fi_{\{{t}_k>s\}}~  \Pi _{j=1}^{k-1} \dd{\Sigma} _{k-1}(t_k)\leq \left(\frac12\right)^{\hat{C}_2T_0^{\frac54}}.
\end{align}
\end{lemma}
\begin{proof}
Take $\eps>0$ sufficiently small. We introduce the following non-grazing set
$$
\CV_{j}^\eps:=\{v_j=(v_{j,1},v_{j,2},v_{j,3})\in \CV_j:\eps \leq |v_{j,1}|\leq \eps^{-1}\}.
$$
We claim that if $\|G\|_{L^{\infty}_{t,x}}\leq 1$, then it holds that $t_{j}-t_{j+1}\geq \f{\eps}{2}$, for $v_{j}\in V_{j}^{\eps}$. In fact, we can solve $[X(\tau),V(\tau)]$ from the characteristic ODE \eqref{2.1.1} as:
\begin{align}\label{3.1.3-2}
X(\tau;t_j,x_j,v_j)&=x_j-\int_{\tau}^{t_j}V_1(\tau';t_j,x_j,v_j)\dd\tau',\nonumber\\
V_1(\tau;t_j,x_j,v_j)&=v_{j,1}-\int_{\tau}^{t_j}G(\tau',X(\tau';t_j,x_j,v_j))\dd\tau'.
\end{align}
Here $t_{j+1}\leq \tau\leq t_j$.

\medskip
\noindent{\it Case 1:} $x_j=0$, $x_{j+1}=1$ or $x_j=1$, $x_{j+1}=0$. Then from \eqref{3.1.3-2} we obtain that
\begin{align}
|(t_{j}-t_{j+1})v_{j,1}|&=\left|x_j-x_{j+1}+\int_{t_{j+1}}^{t_j}G[\tau,X(\tau)](\tau-t_{j+1})\dd\tau\right|\nonumber\\
&\geq |x_{j}-x_{j+1}|-\f{\|G\|_{L^{\infty}_{t,x}}(t_j-t_{j+1})^2}{2}\nonumber.
\end{align}
Then if $t_j-t_{j+1}\leq \f{\eps}{2}$, it holds that
\begin{align}
|(t_{j}-t_{j+1})|\geq |v_{j,1}|^{-1}-\f{\eps^2\|G\|_{L^{\infty}_{t,x}}}{8|v_{j,1}|}\geq \f{7\eps}{8},\nonumber
\end{align}
which leads to a contradiction.

\medskip
\noindent{\it Case 2:} $x_j=x_{j+1}=0$ or $x_j=x_{j+1}=1.$ In this case, there exists $\b{t}_j\in(t_{j+1},t_j)$, such that $V(\b{t}_j;t_j,x_j,v_j)=0$. Therefore, by taking $\tau=\b{t}_j$ in the second equation of \eqref{3.1.3-2}, we have
\begin{align}
\eps{\leq |v_{j,1}|=}  |V_1(\b{t}_j)-v_{j,1}|\leq \int_{t_{j+1}}^{t_j}|G(\tau',X(\tau'))|\dd\tau\leq \|G\|_{L^{\infty}_{t,x}}|t_{j}-t_{j+1}|,\nonumber
\end{align}
which implies that $t_j-t_{j+1}\geq \|G\|_{L^{\infty}}^{-1}\cdot\eps\geq \eps,$ which also leads to a contradiction. This completes the proof of claim.\\

Therefore, if $t_k=t_k(t,x,v,v_1, \cdots, v_{k-1})>0$, there can be at most $\left[ \frac{2T_0}{\epsilon}\right]+1$ number of $v_j\in {\mathcal{V}}_j^\epsilon$ for $1\leq j\leq k-1$. Hence we have
\begin{align}
&\int_{\Pi _{j=1}^{k-1}{\mathcal{V}}_{j}} \Fi_{\{{t}_k>0\}}\dd\Sigma_{k}(t_{k-1})\nonumber\\
&\leq \sum_{n=1}^{\left[ \frac{2T_0}{\epsilon}\right]+1} \int_{\{\mbox{There are n number} ~v_j\in {\mathcal{V}}_j^{\epsilon} ~\mbox{for some} ~1\leq j\leq k-1\}} \dd \Sigma_k(t_{k-1})\nonumber\\
&\leq C\sum_{n=1}^{\left[ \frac{2T_0}{\epsilon}\right]+1} \left(
\begin{gathered}
k-1\nonumber\\
n
\end{gathered}\right) \left|\sup_{j}\int_{{\mathcal{V}}_j^\epsilon}\f{\tilde{w}(v_j)}{\tilde{w}(V_{cl}(t_{j+1}))} \dd{\sigma}_j\right|^n\cdot \left|\sup_{j}\int_{{\mathcal{V}}_j\backslash{\mathcal{V}}_j^\epsilon}\f{\tilde{w}(v_j)}{\tilde{w}(V_{cl}(t_{j+1}))} \dd{\sigma}_j\right|^{k-1-n}.\nonumber
\end{align}
Notice that 
$$
|V_{cl}(t_{j+1})-v_j|\leq \int_{t_{j+1}}^{t_j}|G(\tau,X(\tau))|\dd\tau\leq T_0\|G\|_{L^{\infty}_{t,x}}\ll1,
$$ 
for
$\|G\|_{L^{\infty}_{t,x}}$ sufficiently small. Then it holds that $$\left|\sup_{j}\int_{{\mathcal{V}}_j^\epsilon}\f{\tilde{w}(v_j)}{\tilde{w}(V_{cl}(t_{j+1}))} \dd{\sigma}_j\right|\sim1.$$ Taking $k-2=N\left(\left[ \frac{2 T_0}{\epsilon}\right]+1\right)$, we have
\begin{align}
\int_{\Pi _{j=1}^{k-1}{\mathcal{V}}_{j}} ~\Fi_{\{{t}_k>0\}}~ \dd\Sigma_k(t_{k-1})
&\leq\left\{ 2N\left(\left[ \frac{2 T_0}{\epsilon}\right]+1\right)\right\}^{\left[ \frac{2 T_0}{\epsilon}\right]+2}  (C\epsilon)^{\f{N}{2}\left(2+\left[ \frac{2 T_0}{\epsilon}\right]\right)}\nonumber\\
&\leq \left\{ 4N\left[ \frac{2 T_0}{\epsilon}\right] (C\epsilon)^{\frac{N}{2}}\right\}^{\left[ \frac{2 T_0}{\epsilon}\right]+2} \nonumber\\
&\leq \left\{ C_{N} \cdot T_0\cdot \epsilon^{\frac{N}{2}-1}\right\}^{\left[ \frac{2 T_0}{\epsilon}\right]+2}.\nonumber
\end{align}
Taking $\epsilon=\left(\frac{1}{2C_{N} \cdot T_0}\right)^{\frac{1}{\frac{N}{2}-1}}$ which is small for $T_0$ large, we have  $ C_{N} \cdot T_0\cdot \epsilon^{\frac{N}{2}-1}=\frac12$. Moreover, we note that
$\left[ \frac{2 T_0}{\epsilon}\right]+2\cong C_{N}T_0^{1+\frac{1}{\frac{N}{2}-1}}$ if $T_0$ is large. Finally, we take $N=10$, so that $\left[ \frac{2 T_0}{\epsilon}\right]+2\cong CT_0^{\frac54}$  and $k=10\left\{\left[ \frac{2 T_0}{\epsilon}\right]+1\right\}+2\cong CT_0^{\frac54}$. Then \eqref{3.1.3-1}  follows. Therefore we have completed the proof of Lemma \ref{lm3.2}.
\end{proof}

\begin{proposition}
Let $\beta>3$, $0\leq q<1$ and $\lambda>0$. Assume that for each $i=0,1,2,\cdots $, $h^i(t,x,v)$ is time-periodic function with period $T>0$ and satisfies $$\|h^i\|_{L^\infty_{t,x,v}}+|h^i|_{L^{\infty}_tL^\infty_{\pm}}<\infty.$$ Then there exist
two universal constants $C>0$ and $k\gg1$, independent of $i$ and $\lambda$, such that if $|X_{w}-1|_{C^2}$ sufficiently small, it holds, for $i\geq k$, that
\begin{align}\label{3.1.4}
&\|h^{i+1}\|_{L^\infty_{t,x,v}}+|h^{i+1}|_{L^{\infty}_tL^{\infty}_{\pm}}\nonumber\\
&\quad\leq \frac18 \max_{0\leq l\leq k}\|h^{i-l}\|_{L^\infty_{t,x,v}}+C \max_{0\leq l\leq k}\left\|\f{h^{i-l}}{w}\right\|_{L^2(0,T;L^2)}+C\|\nu^{-1}wg\|_{L^\infty_{t,x,v}}+C|wr|_{L^\infty_tL^{\infty}_{-}}.
\end{align}
Moreover, if $h^i\equiv h$ for $i=1,2,\cdots$, that is $h$ is a solution, then it holds that
\begin{align}\label{3.1.5}
\|h\|_{L^\infty_{t,x,v}}+|h|_{L^\infty_tL^{\infty}_\pm}\leq C \|\nu^{-1}wg\|_{L^\infty_{t,x,v}}+C|wr|_{L^{\infty}_tL^\infty_-}+C\left\|\f{h}{w}\right\|_{L^2(0,T;L^2)}.
\end{align}
\end{proposition}

\begin{proof}
Using Lemma \ref{lm3.1}, we take $s=-nT$ in \eqref{3.1.3} with integer $n\gg1$ sufficiently large and take $T_0=(n+1)T$ and $k=C_1((n+1)T)^{\f54}$ so that \eqref{3.1.3-1} holds for any $(t,x,v)\in \{[0,T]\times (0,1)\times \mathbb{R}^3\}\cup\{[0,T]\times\g_+]\}$. Notice that if $\|G\|_{L^{\infty}_{t,x}}$ sufficiently small, it holds that 
$$
|\tilde{v}(t,x,v)|=|\nu(v)-\f{G(t,x)v_1}{2}-\f{G(t,x)\pa_{v_1}w}{w}|\geq c_1\nu(v)\geq \nu_0>0,
$$ 
for some positive constants $c_1$ and $\nu_0$. Then by a direct computation, we have
\begin{align}\label{3.1.6}
|J_1|\leq Ce^{-\nu_0(t+nT)}\|h^{i+1}\|_{L^{\infty}_{t,x,v}},
\end{align}
\begin{align}\label{3.1.7}
|J_3|\leq C\int_{\max\{t_1,s\}}^te^{-c_1\nu(v)(t-\tau)}\nu(v)\dd\tau\cdot\|\nu^{-1}wg\|_{L^{\infty}_{t,x,v}}\leq C\|\nu^{-1}wg\|_{L^{\infty}_{t,x,v}},
\end{align}
and
\begin{align}\label{3.1.8}
|J_4|\leq C|wr|_{L^{\infty}_tL^{\infty}_-}.
\end{align}
For $J_6$, notice that
\begin{align*}
|V_{cl}(t_{j+1})-v_j|&\leq \int_{t_{j+1}}^{t_j}|G(\tau,X(\tau;t_j,x_j,v_j))|\dd\tau\leq \|G\|_{L^{\infty}_{t,x}}(t_j-t_{j+1})\nonumber\\
&\leq (n+1)T\|G\|_{L^{\infty}_{t,x}}.
\end{align*}
Then if $\|G\|_{L^{\infty}_{t,x}}\ll_{n}1,$ it holds that
$$
\left|\sup_{j}\int_{\CV_j}\f{\tilde{w}(v_j)}{\tilde{w}(V_{cl}(t_{j+1}))}\dd\sigma_j\right|\leq (1+\f{C}{n^{5/4}}),
$$
which implies that \begin{align}\label{3.1.9}
|J_6|&\leq Cn^{5/4}e^{-\nu_0(t+nT)}\max_{1\leq l\leq k-1}\|h^{i+1-l}\|_{L^{\infty}_{t,x,v}}\cdot\left|\sup_{j}\int_{\CV_j}\f{\tilde{w}(v_j)}{\tilde{w}(V_{cl}(t_{j+1}))}\dd\sigma_j\right|^{C_1n^{5/4}}\nonumber\\
&\leq Cn^{5/4}(1+\f{C}{n^{5/4}})^{Cn^{5/4}}e^{-\nu_0(t+nT)}\max_{1\leq l\leq k-1}\|h^{i+1-l}\|_{L^{\infty}_{t,x,v}}\nonumber\\
&\leq Cn^{5/4}e^{-\nu_0(t+nT)}\max_{1\leq l\leq k-1}\|h^{i+1-l}\|_{L^{\infty}_{t,x,v}}.
\end{align}
Similarly, it holds that
\begin{align}\label{3.1.10}
|J_5|\leq Cn^{5/4}|wr|_{L^{\infty}_tL^{\infty}_-},
\end{align}
and
\begin{align}\label{3.1.11}
|J_8|+|J_{10}|\leq Cn^{5/4}\|\nu^{-1}wg\|_{L^{\infty}_{t,x,v}}.
\end{align}
By using \eqref{3.1.3-1}, we have
\begin{align}\label{3.1.12}
|J_{11}|\leq \left(\f12\right)^{\hat{C}_2(nT)^{5/4}}\cdot\max_{1\leq l\leq k-1}\|h^{i+1-l}\|_{L^{\infty}_{t,x,v}}.
\end{align}
Now we consider $J_7$ and $J_9$. Notice that
\begin{align}
|J_7|&\leq  C\sum_{l=1}^{k-1}\int_{\Pi _{j=1}^{l-1}{\mathcal{V}}_{j}}\Pi_{j=1}^{l-1}\f{\tilde{w}(v_j)}{\tilde{w}(V(t_{j+1}))} \dd{\sigma}_j \nonumber\\
&\qquad\times \int_{\mathcal{V}_l}\int_{\mathbb{R}^3} \int_s^{{t}_l} e^{-\nu_0(t_l-\tau)} \Fi_{\{{t}_{l+1}\leq s<{t}_l\}} \tilde{w}(v_l) |k_w(V_{cl}(\tau),v') h^{i-l}(\tau,X_{cl}(\tau),v')|\dd \tau\dd v' \dd{\sigma}_l \nonumber\\
&=C\sum_{l=1}^{k-1}\int_{\Pi _{j=1}^{l-1}{\mathcal{V}}_{j}} \Pi_{j=1}^{l-1}\f{\tilde{w}(v_j)}{\tilde{w}(V(t_{j+1}))} \dd{\sigma}_j \int_{\mathcal{V}_l\cap \{|v_l|\geq N\}}\int_{\mathbb{R}^3}\int_s^{{t}_l} (\cdots)\dd \tau\dd v' \dd{\sigma}_l \nonumber\\
&\quad+C\sum_{l=1}^{k-1}\int_{\Pi _{j=1}^{l-1}{\mathcal{V}}_{j}}  \Pi_{j=1}^{l-1}\f{\tilde{w}(v_j)}{\tilde{w}(V(t_{j+1}))} \dd{\sigma}_j \int_{\mathcal{V}_l\cap \{|v_l|\leq N\}}\int_{\mathbb{R}^3} \int_s^{{t}_l}  (\cdots)\dd \tau\dd v' \dd{\sigma}_l\nonumber\\
&:=\sum_{l=1}^{k-1} (J_{71l}+J_{72l}).\nonumber
\end{align}
For $J_{71l}$, we have
\begin{align}
\sum_{l=1}^{k-1}|J_{71l}|&\leq Cn^{5/4}(1+\f{C}{n^{5/4}})^{Cn^{5/4}}\max_{1\leq l\leq k-1}\|h^{i-l}\|_{L^{\infty}_{t,x,v}}\cdot\int_{|v_l|\geq N}\tilde{w}(v_l)\mu(v_l)(1+|v_l|)\dd v_l\nonumber\\
&\leq \f{Cn^{5/4}}{N}\max_{1\leq l\leq k-1}\|h^{i-l}\|_{L^{\infty}_{t,x,v}}.\nonumber
\end{align}
For $J_{72l}$, we split
\begin{align}
&J_{72l}= C\int_{\Pi _{j=1}^{l-1}{\mathcal{V}}_{j}}\Pi_{j=1}^{l-1}\f{\tilde{w}(v_j)}{\tilde{w}(V(t_{j+1}))} \dd{\sigma}_j\bigg\{\int_{t_l-\f1N}^{t_l}\int_{\CV_{l}\cap\{|v_l|\leq N\}}\int_{\mathbb{R}^3}(\cdots)\dd\tau\dd\sigma_l\dd v'\nonumber\\
&+\int_{s}^{t_l-\f1N}\int_{\CV_{l}\cap\{|v_l|\leq N\}}\int_{|v'|\geq 2N}(\cdots)\dd\tau\dd\sigma_l\dd v'+\int_{s}^{t_l-\f1N}\int_{\CV_{l}\cap\{|v_l|\leq N\}}\int_{|v'|\leq 2N}(\cdots)\dd\tau\dd\sigma_l\dd v'
\bigg\}.\nonumber
\end{align}
Notice that if $|v_{l}|\leq N$, then it holds that
\begin{align}\label{a1}
|V_{cl}(\tau)|&\leq |V_{cl}(\tau)-v_l|+|v_l|\leq|v_l|+\int_\tau^{t_l}\left|G[\tau',X_{cl}(\tau')]\right|\dd\tau'\nonumber\\
&\leq |v_l|+(n+1)T\|G\|_{L^{\infty}_{t,x}}\leq \f{3N}{2}.
\end{align}
Hence by using \eqref{k1}, it holds that $$\int_{|v'|\geq 2N}|k_w(V_{cl}(\tau),v')|\dd v'\leq e^{-\f{N^2}{64}}\int_{|v'|\geq 2N}|k_w(V_{cl}(\tau),v')|e^{\f{|V_{cl}(\tau)-v'|^2}{16}}\dd v'\leq C e^{-\f{N^2}{64}}.
$$
Therefore, the integral in brackets is
\begin{align}
&\leq \f{C}{N}\max_{1\leq l\leq k-1}\|h^{i-l}\|_{L^{\infty}_{t,x,v}}+C\int_{s}^{t_l-\f1N}\int_{\CV_{l}\cap\{|v_l|\leq N\}}\int_{|v'|\leq 2N}(\cdots)\dd\tau\dd\sigma_l\dd v'\nonumber\\
&\leq\f{C}{N}\max_{1\leq l\leq k-1}\|h^{i-l}\|_{L^{\infty}_{t,x,v}}\nonumber\\
&\quad+C\sqrt{\int_{s}^{t_l-\f1N}e^{-\nu_0(t_l-\tau)}\int_{\{|v_l|\leq N\}}\int_{|v'|\leq 2N}e^{-\f{|v_l|^2}{8}}|k_w(V_{cl}(\tau),v')|^2\dd\tau\dd v_l\dd v'}\nonumber\\
&\quad\quad \times\sqrt{\int_{s}^{t_l-\f1N}e^{-\nu_0(t_l-\tau)}\int_{\CV_{l}\cap\{|v_l|\leq N\}}\int_{|v'|\leq 2N}\Fi_{\{t_{l+1}\leq s\leq {t_l}\}}|h^{i-l}(\tau,X_{cl}(\tau),v')|^2\dd\tau\dd\sigma_l\dd v'}\nonumber\\
&\leq \f{C}{N}\max_{1\leq l\leq k-1}\|h^{i-l}\|_{L^{\infty}_{t,x,v}}\nonumber\\
&\quad+\underbrace{C_N\sqrt{\int_{s}^{t_l-\f1N}\int_{\CV_{l}\cap\{|v_l|\leq N\}}\int_{|v'|\leq 2N}\Fi_{\{t_{l+1}\leq s\leq {t_l}\}}|f^{i-l}(\tau,X_{cl}(\tau),v')|^2\dd\tau\dd v_l\dd v'}}_{\Delta}.\nonumber
\end{align}
From the characteristic ODE \eqref{2.1.1}, it holds that
$$
X_{cl}(\tau)=x_{l}-(t_l-\tau)v_{l,1}+\int_{\tau}^{t_l}G(\tau',X_{cl}(\tau'))(\tau'-\tau)\dd\tau'.
$$
Then we have
$$
\f{\pa X_{cl}(\tau)}{\pa_{v_{l,1}}}=-(t_l-\tau)+\int_{\tau}^{t_l}\pa_{x}G(\tau',X_{cl}(\tau'))
\cdot\pa_{v_{l,1}}X_{cl}(\tau')(\tau'-\tau)\dd\tau'.
$$
Since $\|\pa_xG\|_{L^{\infty}_{t,x}}<\infty$, it holds that $|\pa_{v_{l,1}}X_{cl}(\tau')|\leq Ce^{C(t_l-\tau')},$ for some constant $C>0$, which implies that
\begin{align}\label{3.1.13-1}
\left|\f{\pa X_{cl}(\tau)}{\pa_{v_{l,1}}}\right|\geq(t_l-\tau)(1-C\|\pa_xG\|_{L^{\infty}_{t,x}}e^{C(n+1)T})\geq (t_{l}-\tau)(1-C_n\|\pa_xG\|_{L^{\infty}_{t,x}}).
\end{align}
Therefore, if $\|\pa_{x}G\|_{L^{\infty}_{t,x}}\leq C\|\dot{V}_w\|_{L^{\infty}}\ll_n 1$, we have 
\begin{align}\label{3.1.13-2}
\left|\f{\pa X_{cl}(\tau)}{\pa_{v_{l,1}}}\right|\geq \f{t_l-\tau}{2}\geq \f{1}{2N},
\end{align}
for any $s\leq \tau\leq t_l-\f{1}{N}$. Hence we make change of variable $v_{l,1}\rightarrow X_{cl}(\tau)$ in $\Delta$ to deduce that
$$\Delta\leq C_{N,n}\|f^{i-l}\|_{L^{2}(0,T;L^2)}.
$$
So collecting these estimates, we obtain
\begin{align}\label{3.1.14}
|J_7|\leq \f{Cn^{5/4}}{N}\max_{1\leq l\leq k-1}\|h^{i-l}\|_{L^{\infty}_{t,x,v}}+C_{N,n}\max_{1\leq l\leq k-1}\|f^{i-l}\|_{L^{2}(0,T;L^2)}.
\end{align}
Similarly, we have
\begin{align}\label{3.1.15}
|J_9|\leq \f{Cn^{5/4}}{N}\max_{1\leq l\leq k-1}\|h^{i-l}\|_{L^{\infty}_{t,x,v}}+C_{N,n}\max_{1\leq l\leq k-1}\|f^{i-l}\|_{L^{2}(0,T;L^2)}.
\end{align}
Combining \eqref{3.1.3}, \eqref{3.1.6}, \eqref{3.1.7}, \eqref{3.1.8}, \eqref{3.1.9}, \eqref{3.1.10}, \eqref{3.1.11}, \eqref{3.1.12}, \eqref{3.1.14} and \eqref{3.1.15}, we obtain, for any $t\in [0,T]$ and almost every $(x,v)\in [(0,1)\times \mathbb{R}^3]\cup \g_+$, that
\begin{align}\label{3.1.16}
|h^{i+1}(t,x,v)|\leq \int_{\max\{t_1,s\}}^te^{-\nu_0(t-\tau)}\int_{\mathbb{R}^3}\left|k_w(V_{cl}(\tau),v')h^i(\tau,X_{cl}(\tau),v')\right|\dd v'\dd \tau+A_i(t),
\end{align}
where we have denoted
\begin{align}
A_i(t):=&Ce^{-\nu_0(t+nT)}\|h^{i+1}\|_{L^{\infty}_{t,x,v}}+Cn^{5/4}\bigg\{e^{-\nu_0(t+nT)}+\left(\f12\right)^{\hat{C}_2(nT)^{\f54}}+\f1N\bigg\}\
\max_{1\leq l\leq k-1}\|h^{i+1-l}\|_{L^{\infty}_{t,x,v}}\nonumber\\
&+Cn^{5/4}\{\|\nu^{-1}wg\|_{L^{\infty}_{t,x,v}}+|wr|_{L^{\infty}_tL^\infty_-}\}+C_{N,n}\max_{1\leq l\leq k-1}\left\|f^{i-l}\right\|_{L^{2}(0,T;L^2)}.\nonumber
\end{align}
Now applying \eqref{3.1.16} to $h^{i}(\tau,X_{cl}(\tau),v')$, we have
\begin{align}\label{3.1.17}
|h^{i+1}(t,x,v)|\leq& A_i(t)+\int_{\max\{t_1,s\}}^te^{-\nu_0(t-\tau)}\int_{\mathbb{R}^3}|k_w(V_{cl}(\tau),v')A_{i-1}(\tau)\dd v'\dd \tau\nonumber\\
&+\int_{\max\{t_1,s\}}^te^{-\nu_0(t-\tau)}\dd\tau\int_{\mathbb{R}^3}\dd v'\int_{\max\{t_1',s\}}^\tau\int_{\mathbb{R}^3}
U(\tau',v',v'';\tau,v)\dd v''\dd\tau'\nonumber\\
=&A_i(t)+B_1+B_2.
\end{align}
Here we have used the notations $$[t_1',X_{cl}'(\tau'),V_{cl}'(\tau')]=[\tau-t_{\mathbf{b}}(\tau,X_{cl}(\tau),v'),X_{cl}(\tau';\tau,X_{cl}(\tau),v'),V_{cl}(\tau';\tau,X_{cl}(\tau),v')],$$
and
$$
U(\tau',v',v'';\tau,v):=\left|e^{-\nu_0(\tau-\tau')}k_w(V_{cl}(\tau),v')k_w(V_{cl}'(\tau'),v'')h^{i-1}(\tau',X_{cl}'(\tau'),v'')\right|.
$$
By using \eqref{k2}, it is straightforward to verify that
\begin{align}\label{3.1.18}
|B_{1}|\leq &Cn^{5/4}\bigg\{e^{-\nu_0nT}+\left(\f12\right)^{\hat{C}_2(nT)^{\f54}}+\f1N\bigg\}\
\max_{0\leq l\leq k}\|h^{i-l}\|_{L^{\infty}_{t,x,v}}\nonumber\\
&+Cn^{5/4}\{\|\nu^{-1}wg\|_{L^{\infty}_{t,x,v}}+|wr|_{L^{\infty}_{t}L^{\infty}_-}\}+C_{N,n}\max_{0\leq l\leq k}\left\|f^{i-l}\right\|_{L^{2}(0,T;L^2)}.
\end{align}
For $B_2$, we divide it into two cases.

\medskip
\noindent{\it Case 1:} $|v|\geq N$. Similar as before, we have
$$|V_{cl}(\tau)|\geq |v|-(n+1)T\|G\|_{L^{\infty}_{t,x}}\geq \f{N}{2}.$$
Then by using \eqref{k2}, it holds that $|B_2|\leq \f{C}{N}\|h^{i-1}\|_{L^{\infty}_{t,x,v}}$.

\medskip
\noindent{\it Case 2:} $|v|\leq N$. We split the integral domain of $U$ with respect to $\dd\tau'\dd v''\dd v'$ into the following four parts:
\begin{align}
&\{|v'|\geq 2N\}\cup\{|v'|\leq 2N, |v''|>3N\}\cup\{|v'|\leq 2N, |v''|\leq 3N, \tau-\f1N\leq \tau'\leq \tau\}\nonumber\\
&\quad\cup\{|v'|\leq 2N, |v''|\leq 3N, \max\{t_1',s\}\leq \tau'\leq \tau-\f1N\}:=\cup_{i=1}^4\mathcal{O}_i\nonumber.
\end{align}
For $\mathcal{O}_1$, the same as \eqref{a1}, it holds that $|V_{cl}(\tau)-v'|\geq \f{N}{2}$, which implies that
\begin{equation}
|k_w(V_{cl}(\tau),v')|\leq e^{-\f{N^2}{64}}e^{\f{|V_{cl}(\tau)-v'|^2}{16}}|k_w(V_{cl}(\tau),v')|.\nonumber
\end{equation}
By \eqref{k1}, it holds that
\begin{equation}
\int_{\mathbb{R}^3}|k_w(V_{cl}(\tau),v')|e^{\f{|V_{cl}(\tau)-v'|^2}{16}}\leq C.\nonumber
\end{equation}
So that
\begin{align}\label{3.1.19}
\int_{\max\{t_1,s\}}^te^{-\nu_0(t-\tau)}\int_{\mathcal{O}_1}U(\tau',v',v'';\tau,v)\dd v''\dd\tau'\dd v'\dd \tau\leq Ce^{-\f{N^2}{64}}\|h^{i-1}\|_{L^{\infty}_{t,x,v}}.
\end{align}
Similarly,
\begin{align}\label{3.1.20}
\int_{\max\{t_1,s\}}^te^{-\nu_0(t-\tau)}\int_{\mathcal{O}_2}U(\tau',v',v'';\tau,v)\dd v''\dd\tau'\dd v'\dd \tau\leq Ce^{-\f{N^2}{64}}\|h^{i-1}\|_{L^{\infty}_{t,x,v}}.
\end{align}
As for $\mathcal{O}_3$, it is straightforward to obtain that
\begin{align}\label{3.1.21}
\int_{\max\{t_1,s\}}^te^{-\nu_0(t-\tau)}\int_{\mathcal{O}_3}U(\tau',v',v'';\tau,v)\dd v''\dd\tau'\dd v'\dd \tau\leq \f{C}{N}\|h^{i-1}\|_{L^{\infty}_{t,x,v}}.
\end{align}
For $\mathcal{O}_4$, it holds from Holder's inequality that
\begin{align}\label{3.1.22}
\int_{\mathcal{O}_4}&U(\tau',v',v'';\tau,v)\dd v''\dd\tau'\dd v'\nonumber\\
&\leq C_N\sqrt{\int_s^{\tau-\f1N}\int_{|v'|\leq 2N}\int_{|v''|\leq 3N}|k_w(V_{cl}(\tau),v')k_w(V_{cl}'(\tau'),v'')|^2\dd v'\dd v''\dd\tau'}\nonumber\\
&\quad\times\sqrt{\int_s^{\tau-\f1N}\int_{|v'|\leq 2N}\int_{|v''|\leq 3N}
\Fi_{\{\max\{t_1',s\}\leq \tau'\leq \tau\}}|f^{i-1}(\tau',X_{cl}'(\tau'),v'')|^2\dd v'\dd v''\dd\tau'}\nonumber\\
&\leq C_{N}\sqrt{\int_s^{\tau-\f1N}\int_{|v'|\leq 2N}\int_{|v''|\leq 3N}
\Fi_{\{\max\{t_1',s\}\leq \tau'\leq \tau\}}|f^{i-1}(\tau',X_{cl}'(\tau'),v'')|^2\dd v'\dd v''\dd\tau'}.
\end{align}
Similar to obtain \eqref{3.1.13-1} and \eqref{3.1.13-2}, we have, for any $s\leq \tau'\leq \tau-\f{1}{N}$, that
$$
\left|\f{\pa X_{cl}(\tau';\tau,X_{cl}(\tau),v')}{\pa v_{1}'}\right|\geq \f{1}{2N},
$$
provided that $\|\pa_xG\|_{L^{\infty}_{t,x}}\leq C\|\dot{V}_w\|_{L^{\infty}}\ll_{n}1$. Then making change of variable $v'_1\rightarrow X_{cl}'(\tau)$, the R.H.S of \eqref{3.1.22} is bounded as
$$
 C_{N}\sqrt{\int_s^T\left\|f^{i-1}(\tau')\right\|^2_{L^2}\dd\tau'}\leq C_{N,n}\|f^{i-1}\|_{L^2(0,T;L^2)}.
$$
Then it holds that
\begin{align}
\int_{\max\{t_1,s\}}^te^{-\nu_0(t-\tau)}\int_{\mathcal{O}_4}U(\tau',v',v'';\tau,v)\dd v''\dd\tau'\dd v'\dd\tau\leq C_{N,n}\|f^{i-1}\|_{L^2(0,T;L^2)},\nonumber
\end{align}
which, together with \eqref{3.1.19}, \eqref{3.1.20} and \eqref{3.1.21}, yields that
\begin{align}
B_2\leq \f{C}{N}\|h^{i-1}\|_{L^{\infty}_{t,x,v}}+C_{N,n}\|f^{i-1}\|_{L^2(0,T;L^2)}.\nonumber
\end{align}
We combine this with \eqref{3.1.17} and \eqref{3.1.18} to get, for any $t\in [0,T]$ and almost every $(x,v)\in\big[(0,1)\times \mathbb{R}^3\big]\cup\gamma_+$, that
\begin{align}\label{3.1.23}
|h^{i+1}(t,x,v)|\leq &Ce^{-\nu_0 nT}\|h^{i+1}\|_{L^{\infty}_{t,x,v}}+Cn^{5/4}\bigg\{e^{-\nu_0nT}+\left(\f12\right)^{\hat{C}_2(nT)^{\f54}}+\f1N\bigg\}\max_{0\leq l\leq k} \|h^{i-l}\|_{L^{\infty}_{t,x,v}}\nonumber\\
&+Cn^{5/4}\{\|\nu^{-1}wg\|_{L^{\infty}_{t,x,v}}+|wr|_{L^{\infty}_{t}L^{\infty}_-}\}+
C_{N,n}\max_{0\leq l\leq k}\|f^{i-l}\|_{L^{2}(0,T;L^2)}.
\end{align}
Then \eqref{3.1.4} follows from \eqref{3.1.23} by taking both $n$ and $N$ suitably large. \eqref{3.1.5} directly follows from \eqref{3.1.4}. Therefore, the proof of Proposition \ref{prop3.1} is complete.
\end{proof}

\subsection{Approximation solutions}
In this part, we will show the existence of time-periodic solution to the linear problem \eqref{3.0.1} by constructing a sequence of approximation solutions. For reader's convenience, we make an outline of procedure as follows:

\begin{itemize}
  \item[\underline{Step 1.}] In this step, we construct the solution $f^{n,\lambda}$ of the following time-periodic problem with $\lambda>0$:
\begin{equation}\label{3.2.1}\left\{
\begin{aligned}
&\pa_tf^{n,\lambda}+v_1\pa_x f^{n,\lambda}+\sqrt{\mu}^{-1}G(t,x)\pa_{v_1}(\sqrt{\mu}f^{n,\lambda})+\left[\nu(v)+\lambda\right]f^{n,\lambda}=g,\\
&f^{n,\lambda}(t,x,v)|_{\g_-}=(1-\f1n)P_{\g}f^{n,\lambda}+r.
\end{aligned}\right.
\end{equation}
 
\item[\underline{Step 2.}] Pass to the limit $n\rightarrow \infty$ to construct the time-periodic solution $f^{\lambda}$ to the following problem:
\begin{equation}\label{3.2.2}\left\{
\begin{aligned}
&\pa_tf^{\lambda}+v_1\pa_x f^{\lambda}+\sqrt{\mu}^{-1}G(t,x)\pa_{v_1}(\sqrt{\mu}f^{\lambda})+\left[\nu(v)+\lambda\right]f^{\lambda}=g,\\
&f^{\lambda}(t,x,v)|_{\g_-}=P_{\g}f^\lambda+r.
\end{aligned}\right.
\end{equation}

\item[\underline{Step 3.}] In this step, we will show that there exists a unique time-periodic solution $f^{\lambda}$ to the following Boltzmann equation with a penalty
\begin{equation}\label{3.2.3}\left\{
\begin{aligned}
&\pa_tf^\lambda+v_1\pa_x f^\lambda+\sqrt{\mu}^{-1}G(t,x)\pa_{v_1}(\sqrt{\mu}f^\lambda)+\left[\nu(v)+\lambda\right]f^\lambda=Kf^{\lambda}+g,\\
&f^\lambda(t,x,v)|_{\g_-}=P_{\g}f^\lambda+r,
\end{aligned}\right.
\end{equation}
for any $\lambda\geq\lambda_0$, where $\lambda_0>0$ is a suitably large constant. We remark that the zero-mass condition \eqref{3.0.2} is not necessary up to now. 

\item[\underline{Step 4.}] We will use a bootstrap argument to show that the existence of the solution to \eqref{3.2.3} for suitably large $\lambda$ indeed leads to the existence of the solution to \eqref{3.2.3} for $\lambda=0$, which is exact what we desire. In this step, the key point is to establish a uniform-in-$\lambda$ estimates on $\|f^{\lambda}\|_{L^2}$. We will see that the zero-mass condition \eqref{3.0.2} plays a crucial role in establishing such an estimate.   
\end{itemize}

In what follows, we will implement the above procedure step by step.

\begin{lemma}\label{lem3.2.1}
Let $\beta>3$, $0\leq q<1, $ $\lambda>0$ and positive integer $n_0\gg1$ so that $\frac18 (1-\frac2n+\frac{3}{2n^2})^{-\frac{k+1}{2}}\leq \frac12$ for $n\geq n_0$. Assume that $g$ and $r$ are time-periodic functions with period $T>0$ and satisfy
$$\|\nu^{-1}wg\|_{L^\infty_{t,x,v}}+|wr|_{L^\infty_{t}L^{\infty}_-}<\infty,$$ and $\|X_{w}-1\|_{C^2}$ is sufficiently small. Then there exists a unique solution $f^{n,\lambda}$ to \eqref{3.2.1}, which is time-periodic with period $T$, and satisfies
\begin{align}\label{3.2.6}
\|wf^{n,\lambda}\|_{L^{\infty}_{t,x,v}}+|wf^{n,\lambda}|_{L^{\infty}_tL^{\infty}_\pm}\leq C_{n,\lambda}\Big\{ |wr|_{L^\infty_tL^{\infty}_-}+\|\nu^{-1}wg\|_{L^\infty_{t,x,v}} \Big\}.
\end{align}
Here the positive constant $C_{n,\lambda}>0$ depends only on $\lambda$ and $n$.
\end{lemma}
\begin{proof}
We first study the following {\it in-flow} problem:
\begin{equation}\left\{
\begin{aligned}
&\pa_tf+v_1\pa_x f+\sqrt{\mu}^{-1}G(t,x)\pa_{v_1}(\sqrt{\mu}f)+\left[\nu(v)+\lambda\right]f=g,\\
&f(t,x,v)|_{\g_-}=r.\nonumber
\end{aligned}\right.
\end{equation}
Let $h=wf$. Then the equation of $h$ reads as
\begin{equation}\label{3.2.6-2}\left\{
\begin{aligned}
&\pa_th+v_1\pa_x h+G(t,x)\pa_{v_1}h+\tilde{\nu}(t,x,v)h=wg,\\
&h(t,x,v)|_{\g_-}=wr.
\end{aligned}\right.
\end{equation}
We can solve \eqref{3.2.6-2} by the method of characteristics. In fact, let $t\in \mathbb{R}$ and $(x,v)\in [0,1]\times \mathbb{R}^3\setminus \gamma_0\cup\gamma_-$. If $t_{\mathbf{b}}(t,x,v)<\infty$, we define
\begin{align}\label{3.2.6-3}
h(t,x,v)=&e^{-\int_{t-t_{\mathbf{b}}(t,x,v)}^t[\tilde{\nu}(\tau')+\lambda]\dd\tau'}wr(t-t_{\mathbf{b}}(t,x,v),x_{\mathbf{b}}(t,x,v),v_{\mathbf{b}}(t,x,v))\nonumber\\
&+\int_{t-t_{\mathbf{b}}(t,x,v)}^te^{-\int_{\tau}^t[\tilde{\nu}(\tau')+\lambda]\dd\tau'}wg(\tau,X(\tau;t,x,v),V(\tau;t,x,v))\dd\tau.
\end{align}
If $t_{\mathbf{b}}(t,x,v)=+\infty$, we define
\begin{align}\label{3.2.6-4}
h(t,x,v)=\int_{-\infty}^te^{-\int_{\tau}^t[\tilde{\nu}(\tau')+\lambda]\dd\tau'}wg(\tau,X(\tau;t,x,v),V(\tau;t,x,v))\dd\tau.
\end{align}
Now we show that
\begin{equation}
h(t,x,v) \text{ is time-periodic with period }T>0,\nonumber
\end{equation}
and
\begin{align}\label{3.2.6-6}
\|h\|_{L^{\infty}_{t,x,v}}+|h|_{L^{\infty}_tL^\infty_\pm}\leq C\big\{
\|\nu^{-1}wg\|_{L^{\infty}_{t,x,v}}+|wr|_{L^{\infty}_tL^\infty_-}\big\}.
\end{align}
On one hand, if $t_{\mathbf{b}}(t,x,v)<\infty$, then a change of variable $\tau'\rightarrow s'=\tau'-T$ shows that
$$
\begin{aligned}
\int_{t+T-t_{\mathbf{b}}(t+T)}^{t+T}&\tilde{\nu}\left(\tau',X(\tau';t+T,x,v),V(\tau';t+T,x,v)\right)\dd\tau'\\
&=\int_{t-t_{\mathbf{b}}(t)}^{t}\tilde{v}\left(s'+T,X(s'+T;t+T,x,v),V(s'+T;t+T,x,v)\right)\dd s'\\
&=\int_{t-t_{\mathbf{b}}(t)}^{t}\tilde{v}\left(s',X(s';t,x,v),V(s';t,x,v)\right)\dd s',
\end{aligned}
$$
where we have used the periodicity of $\tilde{\nu}$ and Lemma \ref{lm2.1.2}. Therefore, by making changing of variables $\tau'\rightarrow s':=\tau'-T, \tau\rightarrow s:=\tau-T$, we have:
$$\begin{aligned}
h(t+T,x,v)=&e^{-\int_{t+T-t_{\mathbf{b}}(t+T)}^t[\tilde{\nu}(\tau')+\lambda]\dd\tau'}wr(t+T-t_{\mathbf{b}}(t+T),x_{\mathbf{b}}(t+T),
v_{\mathbf{b}}(t+T))\nonumber\\
&+\int_{t+T-t_{\mathbf{b}}(t+T)}^{t+T}e^{-\int_{\tau}^{t+T}[\tilde{\nu}(\tau')+\lambda]\dd\tau'}wg(\tau,X(\tau;t+T,x,v),V(\tau;t+T,x,v))\dd\tau\nonumber\\
=&e^{-\int_{t-t_{\mathbf{b}}(t)}^t[\tilde{\nu}(s')+\lambda]\dd s'}wr(t+T-t_{\mathbf{b}}(t),x_{\mathbf{b}}(t),
v_{\mathbf{b}}(t))\nonumber\\
&+\int_{t-t_{\mathbf{b}}(t)}^{t}e^{-\int_{s}^{t}[\tilde{\nu}(s')+\lambda]\dd s'}wg(s+T,X(s;t,x,v),V(s;t,x,v))\dd s\nonumber\\
=&e^{-\int_{t-t_{\mathbf{b}}(t)}^t[\tilde{\nu}(s')+\lambda]\dd s'}wr(t-t_{\mathbf{b}}(t),x_{\mathbf{b}}(t),
v_{\mathbf{b}}(t))\nonumber\\
&+\int_{t-t_{\mathbf{b}}(t)}^{t}e^{-\int_{s}^{t}[\tilde{\nu}(s')+\lambda]\dd s'}wg(s,X(s;t,x,v),V(s;t,x,v))\dd s\\
=&h(t,x,v).
\end{aligned}
$$
On the other hand, if $t_{\mathbf{b}}(t,x,v)=+\infty$, we have, from Lemma \ref{lm2.1.2}, that $t_{\mathbf{b}}(t+T,x,v)=+\infty$. Hence it holds from changing of variables $\tau'\rightarrow s':=\tau'-T, \tau\rightarrow s:=\tau-T$ that
$$
\begin{aligned}
h(t+T,x,v)&=\int_{-\infty}^{t+T}e^{-\int_{\tau}^{t+T}[\tilde{\nu}(\tau')+\lambda]\dd\tau'}wg(\tau,X(\tau;t+T,x,v),V(\tau;t+T,x,v))\dd\tau\\
&=\int_{-\infty}^{t}e^{-\int_{s}^{t}[\tilde{\nu}(s')+\lambda]\dd s'}wg(s+T,X(s;t,x,v),V(s;t,x,v))\dd s\\
&=\int_{-\infty}^{t}e^{-\int_{s}^{t}[\tilde{\nu}(s')+\lambda]\dd s'}wg(s,X(s;t,x,v),V(s;t,x,v))\dd s=h(t,x,v).
\end{aligned}
$$
This shows that $h$ is time-periodic in $t$ with period $T$. Moreover, $\eqref{3.2.6-6}$ can be directly obtained from the explicit solution formulas \eqref{3.2.6-3} and \eqref{3.2.6-4}.

Now we consider the following approximation sequence with respect to \eqref{3.2.1}:
\begin{equation}\label{3.2.7}\left\{
\begin{aligned}
&\pa_tf^{i+1}+v_1\pa_xf^{i+1}+\sqrt{\mu}^{-1}G(t,x)\pa_{v_1}(\sqrt{\mu}f^{i+1})+\left[\nu(v)+\lambda\right]f^{i+1}=g,\\
&f^{i+1}=(1-\f1n)P_{\g}f^{i}+r,\\
&f^0\equiv0.
\end{aligned}
\right.
\end{equation}
By previous analysis, we know that $\{f^{i}\}_{i\geq 0}$ is well-defined and for each $i \geq 0$, $f^{i}$ is time-periodic with period $T>0$ and $h^{i}=wf^{i}\in L^{\infty}$. Next, we are going to establish uniform-in-$i$ estimate on $f^{i+1}$. Taking the inner product of \eqref{3.2.7} with $f^{i+1}$ over $[0,T]\times (0,1)\times \mathbb{R}^3$ and using the periodicity of $f^{i+1}$, we have
\begin{align}
&\int_0^T\lambda\|f^{i+1}(s)\|_{L^2}^2+\|\nu^{1/2}f^{i+1}(s)\|_{L^2}^2+\f12|f^{i+1}(s)|_{L^2_+}^2\dd s\nonumber\\
&\quad=\f12\int_0^T|f^{i+1}(s)|_{L^2_-}^2\dd s+\int_0^T\langle g+\f{G(t,x)v_1f^{i+1}}{2},f^{i+1}\rangle(s)\dd s.\nonumber
\end{align}
By Cauchy-Schwarz, it holds that
$$
\begin{aligned}
&\left|\int_{0}^T\langle g+\f{G(t,x)v_1f^{i+1}}{2},f^{i+1}\rangle(s)\dd s\right|\\
&\quad\leq \{\eta+C\|G\|_{L^{\infty}_{t,x}}\}\cdot\int_0^T\|\nu^{1/2}f^{i+1}(s)\|_{L^2}^2+C_{\eta}\int_0^T\|\nu^{-1/2}g(s)\|_{L^2}^2\dd s,
\end{aligned}
$$
where $\eta>0$ can be chosen arbitrarily small. From the boundary condition $\eqref{3.2.7}_2$, it holds that
$$\begin{aligned}
\f12\int_0^T|f^{i+1}(s)|_{L^2_-}^2\dd s&\leq \f12\int_0^T|P_{\gamma}f^i(s)+r(s)|_{L^2_-}^2\dd s\nonumber\\
&\leq \f12\left(1-\f2n+\f{3}{2n^2}\right)\int_0^T|f^{i}(s)|_{L^2_+}^2\dd s+C_{n}\int_0^T|r(s)|_{L^2_-}^2\dd s.
\end{aligned}
$$
Hence we obtain, for $\|G\|_{L^{\infty}_{t,x}}\ll 1$, that
\begin{align}
\f{1}{2}&\int_0^T|f^{i+1}(s)|_{L^2_+}^2+\int_0^T\lambda\|f^{i+1}(s)\|_{L^2}^2+\f34\|\nu^{1/2}f^{i+1}(s)\|_{L^{2}}^2\dd s\nonumber
\nonumber\\
&\leq \f{1}{2}(1-\f2n+\f{3}{2n^2})\int_0^T|f^i(s)|_{L^2_+}^2\dd s+C_{n}\int_0^T|r(s)|_{L^2_-}^2\dd s+C_n\int_0^T\|\nu^{-1/2}g(s)\|_{L^2}^2\dd s.\nonumber
\end{align}
To show the convergence of $f^i$, we consider the difference $f^{i+1}-f^i$. By a similar energy estimate, we have
\begin{align}\label{3.2.9}
\f{1}{2}&\int_0^T|[f^{i+1}-f^i](s)|_{L^2_+}^2+\int_0^T\lambda\|[f^{i+1}-f^i](s)\|_{L^2}^2+\f34\|\nu^{1/2}[f^{i+1}-f^i](s)\|_{L^{2}}^2\dd s\nonumber
\nonumber\\
&\leq \f{1}{2}(1-\f2n+\f{3}{2n^2})\int_0^T|[f^i-f^{i-1}](s)|_{L^2_+}^2\dd s\leq \cdots\leq \f{1}{2}(1-\f2n+\f{3}{2n^2})^i\int_0^T|f^1(s)|_{L^2_+}^2\dd s\nonumber\\
&\leq \f{1}{2}(1-\f2n+\f{3}{2n^2})^i\cdot\bigg\{\int_0^T|r(s)|^2_{L^{2}_-}+\|\nu^{-1/2}g(s)\|^2_{L^2}\dd s\bigg\}.
\end{align}
From \eqref{3.2.9}, we have, for $n$ suitably large so that $0<(1-\f2n+\f{3}{2n^2})<1$, that $\{f^i\}_{i=0}^{\infty}$ is a Cauchy sequence in $L^2$. Therefore,
for any $i\geq 0$, the following uniform-in-$i$ estimate holds
\begin{align}\label{3.2.10}
\int_0^T\|\nu^{1/2}f^{i}(s)\|_{L^2}^2+|f^i(s)|_{L^2_+}^2\dd s&\leq C_n \int_0^T|r(s)|^2_{L^{2}_-}+\|\nu^{-1/2}g(s)\|^2_{L^2}\dd s\nonumber\\
&\leq C_n|wr|_{L^\infty_tL^{\infty}_-}^2+C_n\|\nu^{-1}wg\|_{L^{\infty}_{t,x,v}}^2.
\end{align}
Next we establish uniform $L^{\infty}$-estimate. Note that \eqref{3.1.4} is also valid if replacing $1$ with $1-\f1n$ in the boundary condition and constants in \eqref{3.1.4} do not depend on $n$. Then utilizing \eqref{3.1.4}, we obtain that
\begin{align}\label{3.2.11}
&\|h^{i+1}\|_{L^\infty_{t,x,v}}+|h^{i+1}|_{L^{\infty}_tL^{\infty}_\pm}\nonumber\\
&\quad\leq \frac18 \max_{0\leq l\leq k} \|h^{i-l}\|_{L^\infty_{t,x,v}}
+C\Big\{|wr|_{L^\infty_tL^\infty_-}+\|\nu^{-1}wg\|_{L^\infty_{t,x,v}}\Big\}+C\max_{0\leq l\leq k} \|\nu^{1/2}f^{i-l}\|_{L^2([0,T];L^2)}\nonumber\\
&\quad\leq \frac18 \max_{0\leq l\leq k}\|h^{i-l}\|_{L^\infty_{t,x,v}}
+C_n\big\{|wr|_{L^{\infty}_t L^{\infty}_-}+\|\nu^{-1}wg\|_{L^{\infty}_{t,x,v}}\big\}.
\end{align}
Here we have used \eqref{3.2.10} in the last inequality. Applying \eqref{A.1} to \eqref{3.2.11} with
$$a_i=\|h^i\|_{L^\infty_{t,x,v}}+|h^i|_{L^\infty_tL^\infty_\pm}\quad\text{and}\quad D=C_n\big\{|wr|_{L^{\infty}_t L^{\infty}_-}+\|\nu^{-1}wg\|_{L^{\infty}_{t,x,v}}\big\},
$$
we have, for $i\geq k+1$, that
\begin{align}
\|h^{i}\|_{L^{\infty}_{t,x,v}}+|h^{i}|_{L^\infty_tL^{\infty}_\pm}&\leq \f18\max_{1\leq l\leq 2k}\|h^{l}\|_{L^{\infty}_{t,x,v}}+\f{8+k}{7}C_n\big\{|wr|_{L^{\infty}_tL^{\infty}_-}+\|\nu^{-1}wg\|_{L^{\infty}_{t,x,v}}\big\}\nonumber\\
&\leq C_n|wr|_{L^{\infty}_tL^\infty_-}+C_n\|\nu^{-1}wg\|_{L^{\infty}_{t,x,v}},\nonumber
\end{align}
where we have used \eqref{3.2.6-6} for $h=h^i, i=1,\cdots,2k$ in the last inequality. Therefore, we have, for each $i\geq 0$, that
\begin{align}
\|h^{i}\|_{L^{\infty}_{t,x,v}}+|h^i|_{L^{\infty}_tL^\infty_\pm}\leq C_n|wr|_{L^{\infty}_-}+C_n\|\nu^{-1}wg\|_{L^{\infty}_{t,x,v}}.\nonumber
\end{align}
Similarly, applying \eqref{3.1.4} to $h^{i+2}-h^{i+1}$, we get
\begin{align}\label{3.2.14}
&\|h^{i+2}-h^{i+1}\|_{L^{\infty}_{t,x,v}}+|h^{i+2}-h^{i+1}|_{L^\infty_tL^{\infty}_\pm}\nonumber\\
&\leq \f18\max_{0\leq l\leq k}\|h^{i+1-l}-h^{i-l}\|_{L^{\infty}_{t,x,v}}+C\max_{0\leq l\leq k}\{\|\nu^{1/2}[f^{i+1-l}-f^{i-l}]\|_{L^{2}(0,T;L^2)}\}\nonumber\\
&\leq \f18\max_{0\leq l\leq k}\|h^{i+1-l}-h^{i-l}\|_{L^{\infty}_{t,x,v}}+C_n\eta_n^{i-k}\sqrt{\int_0^T|r(s)|^2_{L^{2}_-}+\|\nu^{-1/2}g(s)\|^2_{L^2}\dd s}\nonumber\\
&\leq \f18\max_{0\leq l\leq k}\|h^{i+1-l}-h^{i-l}\|_{L^{\infty}_{t,x,v}}+
C_n\eta_{n}^{i+k+1}\{|wr|_{L^\infty_tL^{\infty}_-}+\|\nu^{-1}wg\|_{L^{\infty}_{t,x,v}}\},
\end{align}
where we have denoted $\eta_n:=\sqrt{1-\f2n+\f3{n^2}}$. Choosing $n$ suitably large so that $\f18\eta_n^{-k-1}\leq\f12$, and then utilizing \eqref{A.1-1} to \eqref{3.2.14}, we obtain, for $i\geq k+1$
\begin{align}
&\|h^{i+2}-h^{i+1}\|_{L^{\infty}_{t,x,v}}+|h^{i+2}-h^{i+1}|_{L^\infty_tL^{\infty}_\pm}\nonumber\\
&\quad\leq \left(\f18\right)^{[\f{i}{k+1}]}\max_{0\leq l\leq 2k+1}\|h^{l}\|_{L^{\infty}_{t,x,v}}+C_n\{|wr|_{L^\infty_tL^{\infty}_-}+\|v^{-1}wg\|_{L^{\infty}_{t,x,v}}\}\cdot\eta_n^i\nonumber\\
&\leq C_n\big\{\left(\f18\right)^{[\f{i}{k+1}]}+\eta_n^i\big\}\{|wr|_{L^{\infty}_tL^\infty_-}+\|v^{-1}wg\|_{L^{\infty}_{t,x,v}}\}.\nonumber
\end{align}
Hence $\{h^{i}\}_{i\geq0}$ is a Cauchy sequence in $L^{\infty}$. Denote $h(t,x,v)$ as the limit function. It is standard to check that $f:=\f{h}{w}$ solves \eqref{3.2.1}, for $n\geq n_0$. The periodicity of $f$ and $L^\infty$-estimate \eqref{3.2.6}  follows from the $L^\infty$-convergence. Therefore, the proof of Lemma \ref{lem3.2.1} is complete.
\end{proof}

The next step  is to show the solvability of \eqref{3.2.2}.

\begin{lemma}\label{lem3.2.2}
Let $\lambda>0$, $0\leq q<1$ and $\beta>3$. Under the same assumption as in Lemma \ref{lem3.2.1}, there exists a unique time-periodic solution $f^\lambda$ to \eqref{3.2.2}. Moreover, $f^\lambda$ satisfies
\begin{align}\label{3.2.17}
\|wf^\lambda\|_{L^\infty_{t,x,v}} +|wf^\lambda|_{L^\infty_tL^{\infty}_\pm} \leq C|wr|_{L^\infty_tL^\infty_-}+C\|\nu^{-1}wg\|_{L^\infty_{t,x,v}}.
\end{align}
\end{lemma}
\begin{proof}
We shall first establish the uniform-in-$n$ estimate on the solution $f^{n,\lambda}$ to \eqref{3.2.1} and then show $h^{n,\lambda}:=wf^{n,\lambda}$ is Cauchy in $L^{\infty}$. Taking inner product of \eqref{3.2.1} with $f^{n,\lambda}$ over $[0,T]\times (0,1)\times \mathbb{R}^3$, we have
\begin{align}
\int_0^T&\lambda\|f^{n,\lambda}(s)\|_{L^2}^2+\f34\|\nu^{1/2}f^{n,\lambda}(s)\|_{L^2}^2+\f12|f^{n,\lambda}(s)|_{L^2_+}^2\dd s-C\|G\|_{L^{\infty}_{t,x}}\cdot\int_0^T\|\nu^{1/2}f^{n,\lambda}(s)\|_{L^2}^2\dd s\nonumber\\
&\leq C\int_0^T\|\nu^{-1/2}g(s)\|_{L^2}^2\dd s+\f12\int_0^T\big|(1-\f1n)P_{\g}f^{n,\lambda}+r\big|_{L^2_-}^2\dd s\nonumber\\
&\leq C\int_0^T\|\nu^{-1/2}g(s)\|_{L^2}^2\dd s+\f{1+\eta}{2}\int_0^T|P_\g f^{n,\lambda}(s)|_{L^{2}_+}^2\dd s+C_\eta\int_0^T|r(s)|^2_{L^{2}_-}\dd s,\nonumber
\end{align}
which implies, for $\|G\|_{L^{\infty}_{t,x}}\ll 1$, that
\begin{align}\label{3.2.18}
\int_0^T&\lambda\|f^{n,\lambda}(s)\|_{L^2}^2+\f12\|\nu^{1/2}f^{n,\lambda}(s)\|_{L^2}^2+\f12|(I-P_\gamma)f^{n,\lambda}(s)|_{L^2_+}^2\dd s\nonumber\\
&\leq C\int_0^T\|\nu^{-1/2}g(s)\|_{L^2}^2\dd s+\f{\eta}{2}\int_0^T|P_\g f^{n,\lambda}(s)|_{L^{2}_+}^2\dd s+C_\eta\int_0^T|r(s)|^2_{L^{2}_-}\dd s\nonumber\\
&\leq C\eta\cdot|h^{n,\lambda}|^2_{L^\infty_tL^{\infty}_+}+C_\eta\|\nu^{-1}wg\|^2_{L^{\infty}_{t,x,v}}+C_\eta|wr|^2_{L^{\infty}_tL^\infty_-},
\end{align}
where $\eta$ can be taken arbitrarily small. Then applying $L^{\infty}$-estimate \eqref{3.1.5} to $h^{n,\lambda}:=wf^{n,\lambda}$ and using \eqref{3.2.18}, we have
\begin{align}\label{3.2.19}
\|h^{n,\lambda}\|_{L^{\infty}_{t,x,v}}+|h^{n,\lambda}|_{L^\infty_tL^{\infty}_\pm}&\leq C\|\nu^{-1}wg\|_{L^\infty_{t,x,v}}+C
|wr|_{L^\infty_tL^\infty_-}+C\|\nu^{1/2}f^{n,\lambda}\|_{L^{2}(0,T;L^2)}\nonumber\\
&\leq C\sqrt{\eta}\cdot|h^{n,\lambda}|_{L^\infty_tL^{\infty}_+}+C_{\eta}\|\nu^{-1}wg\|_{L^\infty_{t,x,v}}+C_\eta|wr|_{L^\infty_tL^\infty_-}\nonumber\\
&\leq C\|\nu^{-1}wg\|_{L^\infty_{t,x,v}}+C|wr|_{L^\infty_tL^\infty_-}.
\end{align}
Here we have taken $\eta>0$ suitably small in the last inequality of \eqref{3.2.19}. To show the convergence, we consider the difference: $h^{n_2,\lambda}-h^{n_1,\lambda}$. Note that $f^{n_2,\lambda}-f^{n_1,\lambda}:=w^{-1}\big(h^{n_2,\lambda}-h^{n_1,\lambda}\big)$ solves
\begin{equation}
\left\{\begin{aligned}
&\pa_t(f^{n_2,\lambda}-f^{n_1,\lambda})+v_1\pa_x(f^{n_2,\lambda}-f^{n_1,\lambda})\\
&\qquad+\sqrt{\mu}^{-1}G(t,x)\pa_{v_1}[\sqrt{\mu}(f^{n_2,\lambda}-f^{n_1,\lambda})]
+[\nu(v)+\lambda](f^{n_2,\lambda}-f^{n_1,\lambda})=0,\\
&(f^{n_2,\lambda}-f^{n_1,\lambda})|_{\gamma_-}=(1-\f{1}{n_2})P_{\g}(f^{n_2,\lambda}-f^{n_1,\lambda})+(\f1{n_1}-\f1{n_2})P_{\g}f^{n_1,\lambda}.\nonumber
\end{aligned}\right.
\end{equation}
Then a similar energy estimate shows that
\begin{align}
\int_0^T&\|\nu^{1/2}(f^{n_2,\lambda}-f^{n_1,\lambda})(s)\|_{L^2}^2\dd s\nonumber\\
 &\leq \eta|h^{n_2,\lambda}-h^{n_1,\lambda}|_{L^\infty_tL^{\infty}_+}^2+C_{\eta}\int_0^T\left|(\f1{n_2}-\f1{n_1})P_{\g}f^{n_1,\lambda}(s)\right|_{L^{2}_-}^2\dd s\nonumber\\
 &\leq \eta|h^{n_2,\lambda}-h^{n_1,\lambda}|_{L^{\infty}_tL^\infty_+}^2+C_{\eta}\bigg(\f{1}{n_1^2}+\f{1}{n_2^2}\bigg)\cdot
 \{\|\nu^{-1}wg\|^2_{L^\infty_{t,x,v}}+|wr|^2_{L^\infty_tL^\infty_-}\},\nonumber
\end{align}
where we have used \eqref{3.2.19} in the last inequality. Again, apply \eqref{3.1.5} to $h^{n_2}-h^{n_1}$, we have
\begin{align}
&\|h^{n_2,\lambda}-h^{n_1,\lambda}\|_{L^{\infty}_{t,x,v}}+|h^{n_2,\lambda}-h^{n_1,\lambda}|_{L^\infty_tL^{\infty}_\pm}\nonumber\\
&\leq C\bigg|w\left(\f1{n_2}-\f1{n_1}\right)P_{\g}f^{n_1,\lambda}\bigg|_{L^\infty_tL^{\infty}_-}+C\|\nu^{1/2}(f^{n_2,\lambda}-f^{n_1,\lambda})\|_{L^2(0,T;L^2)}\nonumber\\
&\leq C\eta\cdot|h^{n_2,\lambda}-h^{n_1,\lambda}|_{L^\infty_tL^{\infty}_+}+C_{\eta}\bigg(\f{1}{n_1}+\f{1}{n_2}\bigg)\cdot
\{\|\nu^{-1}wg\|_{L^\infty_{t,x,v}}+|wr|_{L^\infty_tL^\infty_-}\}\nonumber\\
&\leq C\bigg(\f{1}{n_1}+\f{1}{n_2}\bigg)\cdot
\{\|\nu^{-1}wg\|_{L^\infty_{t,x,v}}+|wr|_{L^\infty_tL^\infty_-}\},\nonumber
\end{align}
where we have taken $\eta>0$ suitably small in the last inequality. 
Hence $h^{n,\lambda}$ is Cauchy in $L^{\infty}$. Denote $h^\lambda(t,x,v)$ as the limit function. It is 
straightforward to check that $f^\lambda:=\f{h^\lambda}{w}$ solves \eqref{3.2.2}. Moreover, since for each $n$, $f^n$ is time-periodic with period $T$, then $f$ is also time-periodic with the same period. The $L^\infty$-estimate \eqref{3.2.17} is the consequence of $L^\infty$-convergence. This proves Lemma \ref{lem3.2.2}.
\end{proof}

\begin{lemma}\label{lem3.2.3}
Let $\beta>3$ and $0\leq q<1$. Under the same assumption as in Lemma \ref{lem3.2.1}, there exists a positive constant $\lambda_0>0$, such that, for any $\lambda\geq \lambda_0$, \eqref{3.2.3} admits a unique time-periodic solution $f^{\lambda}$ with period $T$. Moreover, the solution $f^{\lambda}$ satisfies
\begin{align}
\|wf^{\lambda}\|_{L^\infty_{t,x,v}} +|wf^{\lambda}|_{L^\infty_tL^\infty_\pm} \leq C_{\lambda_0}\big\{ |wr|_{L^\infty_-}+\|\nu^{-1}wg\|_{L^\infty_{t,x,v}} \big\}.\nonumber
\end{align}
\end{lemma}
\begin{proof}
We construct the solution in terms of the following iteration scheme:
\begin{equation}\left\{
\begin{aligned}
&\pa_tf^{n+1}+v_1\pa_x f^{n+1}+\sqrt{\mu}^{-1}G(t,x)\pa_{v_1}(\sqrt{\mu}f^{n+1})+\left[\nu(v)+\lambda\right]f^{n+1}=Kf^{n}+g,\nonumber\\
&f^{n+1}|_{\g_-}=P_{\g}f^{n+1}+r,\\
&f^0\equiv0.
\end{aligned}\right.
\end{equation}
By Lemma \ref{lem3.2.2}, $\{f^{n}\}_{n\geq1}$ is well-defined. To show the convergence, we consider the difference $z^{n+1}=f^{n+1}-f^n$. It is straightforward to verify that $z^{n+1}$ solves
\begin{equation}\label{3.2.21}
\left\{
\begin{aligned}
&\pa_tz^{n+1}+v_1\pa_x z^{n+1}+\sqrt{\mu}^{-1}G(t,x)\pa_{v_1}(\sqrt{\mu}z^{n+1})+\left[\nu(v)+\lambda\right]z^{n+1}=Kz^{n}+g,\\
&z^{n+1}|_{\g_-}=P_{\g}z^{n+1},\\
&z^1= f^1.
\end{aligned}\right.
\end{equation}
Then the same as before, multiplying $z^{n+1}$ on the both side of the first equation in \eqref{3.2.21}, we have
\begin{align}\label{3.2.22}
&\f{1}{2}\f{\dd}{\dd t}\|z^{n+1}\|_{L^2}^2+\f12|(I-P_\g)z^{n+1}|_{L^{2}_+}^2+\lambda\|z^{n+1}\|_{L^2}^2+\|\nu^{1/2}z^{n+1}\|_{L^2}^2\notag\\
&\qquad\qquad=
\langle\f{G(t,x)v_1}{2}z^{n+1}+Kz^{n},z^{n+1}\rangle.
\end{align}
By Cauchy-Schwarz, it holds that
$$\begin{aligned}
|\langle Kz^{n},z^{n+1}\rangle|&\leq \int |k(v,u)z^{n}(u)z^{n+1}(v)|\leq \sqrt{\int|k(v,u)||z^{n}(u)|^2}\cdot\sqrt{\int|k(v,u)||z^{n+1}(u)|^2}\\
& \leq \|z^n\|_{L^2}^2+\hat{C}_1\|z^{n+1}\|_{L^2}^2,
\end{aligned}
$$
for some constant $\hat{C}_1>0$. Recall the fact that
$$|\langle\f{G(t,x)v_1}{2}z^{n+1},z^{n+1}\rangle|\leq C\|G\|_{L^{\infty}_{t,x}}\cdot\|\nu^{1/2}z^{n+1}\|_{L^2}^2.
$$
Then if $\|G\|_{L^\infty_{t,x}}$ sufficiently small, we 
obtain by integrating \eqref{3.2.22} over $[0,T]$ that

$$
\begin{aligned}
(\lambda-\hat{C}_1)\int_0^T\|z^{n+1}(t)\|_{L^2}^2\dd t+\f12\int_0^T\|\nu^{1/2}z^{n+1}(t)\|_{L^2}^2\dd t\leq\int_0^T\|z^{n}(t)\|_{L^2}^2.
\end{aligned}
$$
By iterating over $n$, for $\lambda\geq \lambda_0\geq\hat{C}_1+2,$ we have
\begin{equation}\label{3.2.23}
\int_0^T\|z^{n+1}(t)\|_{L^2}^2\dd s\leq \f{1}{2}\int_0^T\|z^{n}(t)\|_{L^2}^2\dd s\leq \cdots\leq \left(\f12\right)^n\int_0^T\|z^1(t)\|_{L^2}^2\dd t.
\end{equation}
Applying \eqref{3.1.5} to $z^{n+1}$ , we have
\begin{equation}\label{3.2.24}
\begin{aligned}
\|wz^{n+1}\|_{L^{\infty}_{t,x,v}}+|wz^{n+1}|_{L^{\infty}_tL^\infty_{\pm}}&\leq \f{1}{8}\|wz^{n}\|_{L^{\infty}_{t,x,v}}+ C\|z^{n}\|_{L^{2}(0,T;L^2)}\\
&\leq\f{1}{8}\|wz^{n}\|_{L^{\infty}_{t,x,v}}+\left(\f12\right)^{n/2}\|z^1\|_{L^2(0,T;L^2)}\\
&\leq\f{1}{8}\|wz^{n}\|_{L^{\infty}_{t,x,v}}+C\left(\f12\right)^{n/2}\{\|\nu^{-1}wg\|_{L^\infty_{t,x,v}}+|wr|_{L^\infty_tL^\infty_-}\},
\end{aligned}
\end{equation}
where we have used \eqref{3.2.23}. From \eqref{3.2.24}, we conclude that $f^{n}$ is a Cauchy sequence. The solution $f^{\lambda}$ is obtained via taking limit $n\rightarrow\infty$. Then Lemma \ref{lem3.2.3}  follows.
\end{proof}

\noindent{\it Proof of Proposition \ref{prop3.1}:}  We denote $\mathcal{S}^{-1}_{\lambda}$ to be the solution operator of \eqref{3.2.3}, that means, $f=\CS^{-1}_{\lambda}g$ is the unique $L^{\infty}$-solution to \eqref{3.2.3}. By Lemma \ref{lem3.2.3}, $\CS^{-1}_{\lambda_0}$ exists for suitably large $\lambda_0.$ Our idea is to {\it extend} the existence of $\CS^{-1}_{0}$ by a bootstrap argument. The proof is divided into two steps.

\medskip
\noindent{\it Step 1.} Uniform-in-$\lambda$ estimate: Let $0\leq \lambda\leq \lambda_0$. We claim that for any $L^{\infty}$-solution $f^{\lambda}$ to \eqref{3.2.3}, it holds that
\begin{align}\label{3.2.25}
\|wf^{\lambda}\|_{L^{\infty}_{t,x,v}}+|wf^\lambda|_{L^{\infty}_tL^\infty_\pm}\leq C\|\nu^{-1}wg\|_{L^{\infty}_{t,x,v}}+C|wr|_{L^{\infty}_tL^\infty_-},
\end{align}
where the constant $C>0$ is independent of $\lambda.$ Indeed, taking the inner product of \eqref{3.2.3} with $f^{\lambda}$ and using coercivity estimate \eqref{c} and the Cauchy-Schwarz, we have
\begin{align}\label{3.2.26}
\lambda\int_0^T&\|f^{\lambda}(s)\|_{L^2}^2\dd s+\f{c_0}{2}\int_0^T\|\nu^{1/2}(I-P)f^{\lambda}(s)\|_{L^2}^2\dd s+
\f12\int_0^T|(I-P_{\g})f^{\lambda}(s)|_{L^2_+}\dd s\nonumber\\
\leq &\eta\int_0^T|P_{\g}f^{\lambda}(s)|_{L^{2}_+}^2\dd s+C_\eta\int_0^T|r(s)|_{L^{2}_-}^2\dd s\nonumber\\
&+C\{\|G\|_{L^{\infty}_{t,x}}+\eta\}\cdot\int_0^T\|\nu^{1/2}f^{\lambda}(s)\|_{L^2}^2\dd s+C_\eta\int_0^T\|\nu^{-1/2}g(s)\|_{L^2}^2\dd s\nonumber\\
\leq& C\eta\cdot|wf^\lambda|_{L^{\infty}_tL^\infty_+}^2+C_\eta\int_0^T|r(s)|_{L^{2}_-}^2\dd s\nonumber\\
&+C\{\|G\|_{L^{\infty}_{t,x}}+\eta\}\cdot\int_0^T\|\nu^{1/2}f^{\lambda}(s)\|_{L^2}^2\dd s+C_\eta\int_0^T\|\nu^{-1/2}g(s)\|_{L^2}^2\dd s.
\end{align}
Here the projection $P$ is defined in \eqref{P} and the constant $\eta>0$ can be chosen arbitrarily small. Now we estimate the fluid part $Pf^{\lambda}$.
Multiplying $\sqrt{\mu}$ on both sides of \eqref{3.2.3} and integrating over $(0,1)\times \mathbb{R}^3$, we have
\begin{align}
\f{\dd}{\dd t}\langle f^\lambda, \sqrt{\mu}\rangle+\lambda\langle f^\lambda, \sqrt{\mu}\rangle=\langle g, \sqrt{\mu}\rangle+\langle r,\sqrt{\mu}\rangle_{\g_-}=0,\nonumber
\end{align}
which implies that
$$\langle f^{\lambda}(t),\sqrt{\mu}\rangle=e^{-\lambda t}\langle f^{\lambda}(0),\sqrt{\mu}\rangle.
$$
Since $f^{\lambda}(t)$ is periodic in $t$, then $\langle f^{\lambda}(t),\sqrt{\mu}\rangle\equiv0$. Then similar as in \cite[Lemma 13]{CKL}, one can find a function $\mathfrak{e}_{f^{\lambda}}(t)\lesssim \|f^\lambda(t)\|_{L^2}^2$, so that
\begin{align}\label{3.2.27}
&\int_0^t\|\nu^{1/2}Pf^{\lambda}(s)\|_{L^2}^2\dd s\leq
\bigg(\mathfrak{e}_{f^{\lambda}}(t)-\mathfrak{e}_{f^{\lambda}}(0)\bigg)+C\int_0^t\|\nu^{1/2}(I-P)f^{\lambda}(s)\|_{L^2}^2\dd s\nonumber\\
&\quad+C\|G\|_{L^{\infty}_{t,x}}\cdot\int_0^t\|\nu^{1/2}f^\lambda(s)\|_{L^2}^2\dd s+C\int_0^t\|\nu^{-1/2}g(s)\|_{L^2}^2+|r(s)|^2_{L^2_-}+|(I-P_{\g})f(s)|_{L^2_+}^2\dd s.
\end{align}
In particular, taking $t=T$ in \eqref{3.2.27} and utilizing the periodicity of $f^{\lambda}$, we obtain
\begin{align}
&\int_0^T\|\nu^{1/2}Pf^{\lambda}(s)\|_{L^2}^2\dd s\leq C\int_0^T\|\nu^{1/2}(I-P)f^{\lambda}(s)\|_{L^2}^2\dd s+C\int_0^T\|\nu^{-1/2}g(s)\|_{L^2}^2\dd s\nonumber\\
&\qquad+C\|G\|_{L^{\infty}_{t,x}}\cdot\int_0^T\|\nu^{1/2}f^\lambda(s)\|_{L^2}^2\dd s+C\int_0^T|r(s)|_{L^2_-}^2\dd s+C\int_0^T|(I-P_{\g})f^\lambda(s)|_{L^2_+}^2\dd s.\nonumber
\end{align}
Combining this with \eqref{3.2.26}, we have, for sufficiently small $\|G\|_{L^{\infty}_{t,x}}$, that
\begin{align}\label{3.2.28}
&\int_0^T\|\nu^{1/2}f^{\lambda}(s)\|_{L^2}^2\dd s+\int_0^T|(I-P_\g)f^\lambda(s)|_{L^2_+}^2\dd s\nonumber\\
&\qquad\leq C\eta\cdot|wf^\lambda|_{L^{\infty}_tL^\infty_\pm}^2+C_\eta\int_0^T|r(s)|_{L^{2}_-}^2+\|\nu^{-1/2}g(s)\|_{L^2}^2\dd s\nonumber\\
&\qquad\leq  C\eta\cdot|wf^\lambda|_{L^\infty_{t}L^{\infty}_\pm}^2+C_\eta|wr|_{L^{\infty}_tL^{\infty}_{-}}^2+C_\eta\|\nu^{-1}wg\|_{L^{\infty}_{t,x,v}}^2.
\end{align}
Applying $L^{\infty}$-estimates \eqref{3.1.5} to $h^{\lambda}=wf^{\lambda}$, we have
\begin{align}
\|h^\lambda\|_{L^\infty_{t,x,v}}+|h^\lambda|_{L^{\infty}_tL^\infty_{\pm}}&\leq C \|\nu^{-1}wg\|_{L^\infty_{t,x,v}}+C|wr|_{L^\infty_tL^\infty_-}+C\|f^\lambda\|_{L^2(0,T;L^2)}\nonumber\\
&\leq C\sqrt{\eta}\cdot|h^\lambda|_{L^{\infty}_tL^\infty_\pm}+ C_\eta\|\nu^{-1}wg\|_{L^\infty_{t,x,v}}+C_\eta|wr|_{L^\infty_tL^\infty_-}\nonumber\\
&\leq C\|\nu^{-1}wg\|_{L^\infty_{t,x,v}}+C|wr|_{L^\infty_tL^\infty_-}.\nonumber
\end{align}
Here we have used \eqref{3.2.28} in the second inequality and taken $\eta>0$ suitably small in the last inequality. This shows the claim \eqref{3.2.25}.

\medskip
\noindent{\it Step 2.} In this step, we shall prove the existence of solution $f^\lambda$ to \eqref{3.2.3}, for $\lambda>0$ sufficiently close to $\lambda_0.$  Firstly, we define an Banach space
\begin{align}\nonumber
\mathbf{X}:=\Big\{f=f(t,x,v) :&\ f(t+T)=f(t),\ wf\in L^\infty_{t,x,v}, \ wf\in L^\infty_tL^\infty_\pm,\nonumber\\
 &\  \mbox{and} \  f|_{\gamma_-}=P_\gamma f+r \Big\},\nonumber
\end{align}
and define the operator:
\begin{align}\nonumber
T_\lambda f=\mathcal{S}_{\lambda_0}^{-1}[(\lambda_0-\lambda) f+g],\quad\text{for}\quad 0\leq \lambda\leq\lambda_0.
\end{align}
Notice that $S_{\lambda_0}^{-1}$ is well-defined by Lemma \ref{lem3.2.3}. Now for any $f_1, f_2\in \mathbf{X}$, by using the uniform estimates \eqref{3.2.25}, we have that
\begin{align}
\|w(T_\lambda f_1-T_\lambda f_2)\|_{L^\infty_{t,x,v}}&= \left\|w\{\mathcal{S}_{\lambda_0}^{-1}[(\lambda_0-\lambda) f_1+g]-\mathcal{S}_{\lambda_0}^{-1}[(\lambda_0-\lambda)f_2+g]\} \right\|_{L^\infty_{t,x,v}}\nonumber\\
&\leq (\lambda_0-\lambda)\| w\CS_{\lambda_0}^{-1}(f_1-f_2)\|_{L^\infty_{t,x,v}}\nonumber\\
&\leq C(\lambda_0-\lambda)\| w(f_1-f_2)\|_{L^\infty_{t,x,v}}.\nonumber
\end{align}
Here the universal constant $C>0$ is independent of $\lambda.$ Taking $0<\lambda_1<\lambda_0$ to be sufficiently close to $\lambda_0$ so that $C(\lambda_0-\lambda_1)\leq \frac12$, then $T_\lambda : \mathbf{X}\rightarrow \mathbf{X}$ is a contraction mapping for $\lambda\in[\lambda_1,\lambda_0]$. Thus $T_\lambda$ has a fixed point, i.e.,  $\exists f^\lambda\in \mathbf{X}$ so that
\begin{equation}\nonumber
f^\lambda=T_\lambda f^\lambda=\mathcal{S}_{\lambda_0}^{-1} \Big[(\lambda_0-\lambda )f^\lambda+g\Big],
\end{equation}
which yields immediately that
\begin{align}\nonumber
\mathcal{S}_\lambda f^\lambda=\pa_tf^\lambda+v_1\pa_x f^\lambda+\sqrt{\mu}^{-1}G(t,x)\pa_{v_1}(\sqrt{\mu}f^\lambda)+\left[\nu(v)+\lambda\right]f^\lambda-Kf^{\lambda}=g.
\end{align}
Therefore we have obtained the existence of $\mathcal{S}_\lambda^{-1}$ for $\lambda\in[\lambda_1,\lambda_0]$. \\

\medskip
\noindent{\it Step 3.} Next we define
\begin{align}\nonumber
T_{\lambda,\lambda_1}f=\mathcal{S}_{\lambda_1}^{-1}\Big[(\lambda_1-\lambda) f+g\Big].
\end{align}
Notice that by the uniform-in-$\lambda$ estimates in \eqref{3.2.25}, the estimates for $\mathcal{S}_{\lambda_1}^{-1}$ are independent of $\lambda_1$. By similar arguments, we can prove $T_{\lambda,\lambda_1} : \mathbf{X}\rightarrow \mathbf{X}$ is a contraction mapping  for $\lambda\in[2\lambda_1-\lambda_0,\lambda_1]$.  Then we obtain the exitence of operator $\mathcal{S}_{2\lambda_1-\lambda_0}^{-1}$.  Step by step, we can finally obtain the existence of operator $\mathcal{S}_0^{-1}$, and the solution $f:=\mathcal{S}_0^{-1}g$ satisfies the estimates in \eqref{3.2.25}. Therefore we complete the proof of Proposition \ref{prop3.1}. \qed

\subsection{Proof of Theorem \ref{thm1.2}}
The solution is constructed via the following iteration scheme:
\begin{align}
\begin{cases}
\pa_t f^{j+1}+v_1\pa_x f^{j+1}+\f{1}{\sqrt{\mu}}G(t,x)\pa_{v_1}({\sqrt{\mu}f^{j+1}})+Lf^{j+1}=\Gamma(f^j,f^j)+G(t,x)v_1\sqrt{\mu},\\[1.5mm]
f^{j+1}|_{\gamma_-}=P_\gamma f^{j+1}+\frac{\sqrt{2\pi}(\mu_{w}-\mu)}{\sqrt{\mu}}\int_{v'\cdot n(x)>0}f^j\sqrt{\mu} \{v'\cdot n(x)\}\dd v'+\f{\mu_w-\mu}{\sqrt{\mu}} ,\nonumber
\end{cases}
\end{align}
for $j=0,1,2\cdots$ with $f^0\equiv0$. A direct calculation shows that
\begin{align}
\int_0^1\int_{\mathbb{R}^3}G(t,x)v_1\sqrt{\mu}\dd v\dd x=\int_0^1\int_{\mathbb{R}^3}\Gamma(f^j,f^j)\sqrt{\mu(v)} \dd v\dd x= \int_{v_1\gtrless0} [\mu_w(t,v)-\mu(v)] v_1 \dd v =0,\nonumber
\end{align}
and
\begin{align}
\|wGv_1\sqrt{\mu}\|_{L^\infty_{t,x,v}}+\left| \sqrt{2\pi}w\f{\mu_w-\mu}{\sqrt{\mu}}\bigg\{1+\int_{v'\cdot n(x)>0} f^j\sqrt{\mu} \{v'\cdot n(x)\} \dd v'\bigg\} \right|_{L^\infty_\pm}\leq C\delta+C\delta |f^j|_{L^\infty_+}.\nonumber
\end{align}
Then applying Proposition \ref{prop3.1} to $f^{j+1}$ with
$g=\Gamma(f^j,f^j)+Gv_1\sqrt{\mu}$ and
$$r= \sqrt{2\pi}\frac{\mu_w-\mu}{\sqrt{\mu}}\bigg\{1+\int_{v'\cdot n(x)>0} f^j\sqrt{\mu} \{v'\cdot n(x)\}\dd v'\bigg\},
$$
and using \eqref{g}, we have
\begin{align}
\|wf^{j+1}\|_{L^\infty_{t,x,v}}+|wf^{j+1}|_{L^\infty_tL^\infty_\pm}\leq C\delta+C\|wf^j\|_{L^{\infty}_{t,x,v}}^2+ C\delta|wf^j|_{L^{\infty}_tL^\infty_+}.\nonumber
\end{align}
Hence it follows from a standard induction argument that
\begin{equation}\label{3.3.7}
\|wf^{j}\|_{L^\infty_{t,x,v}}+|wf^{j}|_{L^\infty_tL^\infty_\pm}\leq 2C\delta,\quad\mbox{for} \  j=1,2,\cdots,
\end{equation}
provided that $\delta>0$ is suitably small. For the convergence of the sequence $f^j$, we consider the difference $g^{j+1}:=f^{j+1}-f^j$. Since $g^{j+1}$ solves
\begin{align}
\pa_tg^{j+1}+v_1 \pa_xg^{j+1}+\f1{\sqrt{\mu}}G(t,x)\cdot\pa_{v_1}(\sqrt{\mu}g^{j+1})+Lg^{j+1}
=\Gamma(g^j,f^{j})+\Gamma(f^{j-1},g^j),\nonumber
\end{align}
with the boundary condition
\begin{align}
g^{j+1}|_{\gamma_-}=P_\gamma g^{j+1}+\frac{\sqrt{2\pi}(\mu_w-\mu)}{\sqrt{\mu}}\int_{v'\cdot n(x)>0}g^j\sqrt{\mu} \{v'\cdot n(x)\} \dd v',\nonumber
\end{align}
we again apply \eqref{3.0.3} to $g^{j+1}$ and get
\begin{align}
&\|wg^{j+1}\|_{L^\infty_{t,x,v}}+|wg^{j+1}|_{L^\infty_tL^\infty_\pm}\nonumber\\
&\qquad\leq C\{\delta+\|wf^{j}\|_{L^\infty_{t,x,v}}+\|wf^{j-1}\|_{L^\infty_{t,x,v}}\}\times \big\{\|wg^{j}\|_{L^\infty_{t,x,v}}+|wg^j|_{L^\infty_tL^\infty_+}\big\}\nonumber\\
&\qquad\leq C\delta\{\|wg^j\|_{L^\infty_{t,x,v}}+|wg^j|_{L^\infty_tL^\infty_+}\},\nonumber
\end{align}
where we have used \eqref{3.3.7} in the last inequality. By taking $\delta$ suitably small, we conclude that $wf^j$ is a Cauchy sequence in $L^\infty$. The time-periodic solution is obtained by taking the limit $j\rightarrow \infty$. The $L^\infty$-estimate \eqref{3.0.5} is the consequence of $L^\infty$-convergence. The uniqueness is standard. Therefore we have completed the proof of Theorem \ref{thm1.2}. \qed

\section{Non-negativity}\label{sec4}
In this section, we show that the solution obtained by Theorem \ref{thm1.2} is nonnegative. The strategy is to show that it is a large time limit of the solution of the Cauchy problem \eqref{1.2.2}, \eqref{1.2.3} with the initial data $F(0,x,v)=F_0(x,v)$. We introduce the perturbation:
$$F(t,x,v)=F^{\text{per}}(t,x,v)+\sqrt{\mu}f(t,x,v).
$$
The equation of $f$ reads as
\begin{align}\label{4.1}
\pa_tf+v_1\pa_xf+\f1{\sqrt{\mu}}G(t,x)\pa_{v_1}(\sqrt{\mu}f)+Lf=\Gamma(f^{\text{per}}+f,f)+\Gamma(f,f^{\text{per}}),
\end{align}
with boundary condition:
\begin{align}\label{4.2}
f|_{\g_-}=P_\g f+\frac{\sqrt{2\pi}(\mu_{w}-\mu)}{\sqrt{\mu}}\int_{v'\cdot n(x)>0}f\sqrt{\mu} \{v'\cdot n(x)\}\dd v',
\end{align}
and initial data
\begin{align}\label{4.2-1}
f(t,x,v)|_{t=0}=f_0(x,v):=\f{F(0,x,v)-F^{\text{per}}(0,x,v)}{\sqrt{\mu}}.
\end{align}
As before, we first consider the linear problem
\begin{equation}\label{4.3}\left\{
\begin{aligned}
&\pa_tf+v_1\pa_xf+\f1{\sqrt{\mu}}G(t,x)\pa_{v_1}(\sqrt{\mu}f)+Lf=g,\\
&f|_{\g_-}=P_\g f+r,\\
&f(t,x,v)|_{t=0}=f_0(x,v).
\end{aligned}\right.
\end{equation}

\begin{proposition}\label{prop4.1}
Let $\beta>3$ and $0\leq q<1$. There exists constant $\lambda_1>0$, such that if
$$\|X_w-1\|_{C^2}\ll 1,\quad\|wf_0\|_{L^{\infty}}+\sup_{s\geq 0}|wr(s)|_{L^{\infty}_-}+\sup_{s\geq0}\|\nu^{-1}wg(s)\|_{L^{\infty}}<\infty,$$
\begin{align}\label{4.4}
P_{\g}r\equiv0,\quad\text{and}\quad\int_0^1\int_{\mathbb{R}^3} f_0(x,v)\sqrt{\mu(v)}\dd x\dd v=\int_0^1\int_{\mathbb{R}^3} g(t,x,v)\sqrt{\mu(v)}\dd x\dd v=0,
\end{align}
then the linear IBVP problem \eqref{4.3} admits a unique solution $f(t,x,v)$ which satisfies, for any $t>0,$ that
\begin{align}\label{4.5}
&\sup_{0\leq s \leq t}\{\|we^{\lambda_1s}f(s)\|_{L^{\infty}}+|we^{\lambda_1s}f(s)|_{L^{\infty}_\pm}\}\nonumber\\
&\qquad\leq C\|wf_0\|_{L^{\infty}}+C\sup_{0\leq s\leq t }\{\|\nu^{-1}we^{\lambda_1 s}g(s)\|_{L^{\infty}}+|we^{\lambda_1 s}r(s)|_{L^{\infty}_-}\}.
\end{align}
Here the positive constant $C$ is independent of $t>0$.
\end{proposition}
\begin{proof}
The proof of existence of the solution is standard, for instance, see \cite[Proposition 3.8]{EGKM2}. We only establish the decay estimates \eqref{4.5}. The proof is divided into several steps:

\medskip
\noindent{\it Step 1.} $L^2$-estimate. Taking the inner product of \eqref{4.3} with $f$, we have
\begin{align}\label{4.6}
\f12\f{\dd}{\dd t}\|f(t)\|_{L^2}^2+\f12|f|_{L^2_+}^2+\langle \sqrt{\mu}^{-1}G\pa_{v_1}(\sqrt{\mu}f),f\rangle+\langle Lf,f\rangle=\f12|f|^2_{L^2_-}+\langle g,f\rangle.
\end{align}
The same as before, we have
\begin{align}\label{4.7}
\f12|f|_{L^2_+}^2-\f12|f|_{L^2_-}^2&=\f12|P_\g f|_{L^2_+}^2+\f{1}{2}|(I-P_\g)f|_{L^2_+}^2-\f12|P_{\g}f|_{L^2_-}^2-\f12|r|^2_{L^2_-}\nonumber\\
&=\f{1}{2}|(I-P_\g)f|_{L^2_+}^2-\f12|r|^2_{L^2_-}.
\end{align}
By the coercivity estimate, it holds that
\begin{align}\label{4.8}
\langle Lf,f\rangle\geq c_0\|\nu^{1/2}(I-P)f\|_{L^2}^2.
\end{align}
A direct computation shows that
\begin{align}\label{4.9}
|\langle \sqrt{\mu}^{-1}G\pa_{v_1}(\sqrt{\mu}f),f\rangle|+|\langle g,f\rangle|&=\f12|\langle Gv_1f,f\rangle|+|\langle g,f\rangle|\nonumber\\
&\leq C\{\|G\|_{L^{\infty}_{t,x}}+\eta\}\cdot\|\nu^{1/2}f\|_{L^2}^2+C_{\eta}\|\nu^{-1/2}g\|_{L^2}^2.
\end{align}
Here the constant $\eta>0$ can be chosen arbitrarily small. For the estimates on the fluid part $Pf$, we multiply $\sqrt{\mu}$ on both side of the first equation in \eqref{4.3} and use \eqref{4.4} to get
\begin{align}
\f{\dd }{\dd t}\langle f,\sqrt{\mu}\rangle=\langle g,\sqrt{\mu}\rangle+\langle r,\sqrt{\mu}\rangle_{\g_-}=0,\nonumber
\end{align}
which implies, for all $t>0$, that
$$\langle f(t),\sqrt{\mu}\rangle=\langle f_0,\sqrt{\mu}\rangle=0.
$$
Therefore, the same as for obtaining \eqref{3.2.27}, one can find a function $\mathfrak{e}(t)\lesssim \|f(t)\|_{L^2}^2$, such that
\begin{align}\label{4.10}
\int_0^t\|\nu^{1/2}Pf(s)\|_{L^2}^2\dd s\leq& \mathfrak{e}(t)-\mathfrak{e}(0)+C\int_0^t\|\nu^{1/2}(I-P)(s)\|_{L^2}^2\dd s
+\int_0^tC\|\nu^{-1/2}g(s)\|_{L^2}^2\dd s\nonumber\\
&+C\int_0^t|r(s)|^2_{L^2_-}+|(I-P_{\g})f(s)|_{L^2_+}^2\dd s+C\|G\|_{L^{\infty}_{t,x}}\cdot\int_0^t\|\nu^{1/2}f(s)\|_{L^2}^2\dd s.
\end{align}
Combining \eqref{4.6}, \eqref{4.7}, \eqref{4.8}, \eqref{4.9} and \eqref{4.10}, we have, for sufficiently small $\|G\|_{L^{\infty}_{t,x}}$, that
\begin{align}\label{4.12}
\|f(t)\|_{L^2}^2+\int_0^t\|\nu^{1/2}f(s)\|_{L^2}^2\dd s\leq C\|f_0\|_{L^2}^2+C\int_{0}^t\|\nu^{-1/2}g(s)\|_{L^2}^2\dd s+C\int_{0}^t|r(s)|_{L^2_-}^2\dd s.
\end{align}

\medskip
\noindent{\it Step 2.} $L^{\infty}$-estimate. In this step, we will prove the following

\medskip
\noindent{\it Claim:} Assume that $|X_w-1|_{C^2}$ is sufficiently small. It holds that
\begin{align}\label{4.12-1}
&\sup_{0\leq s\leq t}\|wf(s)\|_{L^{\infty}}+\sup_{0\leq s\leq t}|wf(s)|_{L^{\infty}_\pm}\nonumber\\
&\qquad\leq C\|wf_0\|_{L^{\infty}}+ C\sup_{0\leq s\leq t}\|\nu^{-1}wg(s)\|_{L^\infty}+C\sup_{0\leq s\leq t}|wr(s)|_{L^\infty_-}+\sup_{0\leq s\leq t}\|f(s)\|_{L^2}.
\end{align}

\noindent{\it Proof of Claim:} Fix $s\geq 0$. Take $T_0$ sufficiently large and $\|G\|_{L^\infty_{t,x}}$ suitably small so that Lemma \ref{lm3.2} holds. Then similar as \eqref{3.1.23}, it holds, for $t\in[s,s+T_0]$, that
\begin{align}\label{4.13}
&\|wf(t)\|_{L^\infty}+|wf(t)|_{L^{\infty}_\pm}\nonumber\\
&\quad\leq CT_0^{5/2}e^{-\nu_0(t-s)}\|wf(s)\|_{L^{\infty}}+CT_0^{5/4}\bigg\{\left(\f12\right)^{\hat{C}_2T_0^{\f54}}+\f1N\bigg\} \sup_{s\leq \tau\leq t}\|wf(\tau)\|_{L^{\infty}}\nonumber\\
&\qquad+CT_0^{5/2}\sup_{s\leq \tau\leq t}\{ \|\nu^{-1}wg(\tau)\|_{L^\infty}+|wr(\tau)|_{L^\infty_-}\}+C_{T_0}\sqrt{\int_s^{t}\left\|f(\tau)\right\|_{L^2}^2\dd \tau}\nonumber\\
&\quad\leq CT_0^{5/2}e^{-\nu_0(t-s)}\|wf(s)\|_{L^{\infty}}+CT_0^{5/4}\bigg\{\left(\f12\right)^{\hat{C}_2T_0^{\f54}}+\f1N\bigg\} \sup_{s\leq \tau\leq t}\|wf(\tau)\|_{L^{\infty}}\nonumber\\
&\qquad+CT_0^{5/2}\sup_{s\leq \tau\leq t}\{ \|\nu^{-1}wg(\tau)\|_{L^\infty}+|wr(\tau)|_{L^\infty_-}\}+C_{T_0}\sup_{s\leq \tau\leq t}\|f(\tau)\|_{L^2}.
\end{align}
Here the positive constant $C$ is independent of $T_0$ and $s$. Denote
$$\begin{aligned}
\CD(s):=&CT_0^{5/4}\bigg\{\left(\f12\right)^{\hat{C}_2T_0^{\f54}}+\f1N\bigg\} \sup_{s\leq \tau\leq s+T_0}\|wf(\tau)\|_{L^{\infty}}\\
&+CT_0^{5/2}\sup_{s\leq \tau\leq s+T_0}\{ \|\nu^{-1}wg(\tau)\|_{L^\infty}+|wr(\tau)|_{L^\infty_-}\}+C_{T_0}\sup_{s\leq \tau\leq s+T_0}\|f(\tau)\|_{L^2}.
\end{aligned}$$
Now for any $t>0$, there exists a positive integer $k$, such that $kT_0\leq t<(k+1)T_0$. Applying \eqref{4.13} to  $f(T_0)$, $f(2T_0),\cdots,$ $f(kT_0)$ inductively, we have:
\begin{align}\label{4.14}
\|wf(kT_0)\|_{L^{\infty}}&\leq CT_0^{5/2}e^{-\nu_0T_0}\|wf[(k-1)T_0]\|_{L^\infty}+\CD[(k-1)T_0]\nonumber\\
&\leq e^{-\f{\nu_0}{2}T_0}\|wf[(k-1)T_0]\|_{L^\infty}+\CD[(k-1)T_0]\nonumber\\
&\leq e^{-2\f{\nu_0}{2}T_0}\|wf[(k-2)T_0]\|_{L^\infty}+\CD[(k-1)T_0]+e^{-\f{\nu_0}{2}T_0}\CD[(k-2)T_0]\nonumber\\
&\leq \cdots\leq e^{-\f{k\nu_0T_0}{2}}\|wf_0\|_{L^{\infty}}+\sum_{i=0}^{k-1}e^{-\f{(k-1-i)\nu_0}{2}T_0}D(iT_0)\nonumber\\
&\leq e^{-\f{k\nu_0T_0}{2}}\|wf_0\|_{L^{\infty}}+CT_0^{5/4}\bigg\{\left(\f12\right)^{\hat{C}_2T_0^{\f54}}+\f1N\bigg\}\cdot\sup_{0\leq\tau\leq t}\|wf(\tau)\|_{L^{\infty}}\nonumber\\
&\quad+CT_0^{5/4}\sup_{0\leq\tau\leq t}\cdot\{\|\nu^{-1}wg(\tau)\|_{L^\infty}+|wr(\tau)|_{L^\infty_-}\}+C_{T_0}\cdot\sup_{0\leq \tau\leq t}\|f(\tau)\|_{L^2},
\end{align}
for suitably large $T_0>0$. Taking $s=kT_0$ in \eqref{4.13} and substituting \eqref{4.14} into the resultant equation, we have
\begin{align}
&\sup_{0\leq s\leq t}\|wf(s)\|_{L^{\infty}}+\sup_{0\leq s\leq t}\|wf(s)\|_{L^{\infty}_\pm}\nonumber\\
&\quad\leq C\|wf_0\|_{L^{\infty}}+CT_0^{5/4}\bigg\{\left(\f12\right)^{\hat{C}_2T_0^{\f54}}+\f1N\bigg\} \cdot\sup_{0\leq s\leq t}\|wf(s)\|_{L^{\infty}}\nonumber\\
&\qquad+CT_0^{5/4}\cdot\sup_{0\leq s\leq t}\{ \|\nu^{-1}wg(s)\|_{L^\infty}+|wr(s)|_{L^\infty_-}\}+C_{T_0}\cdot\sup_{0\leq s\leq t}\|f(s)\|_{L^2}.\nonumber
\end{align}
Then \eqref{4.12-1} follows from taking both $T_0$ and $N$ suitably large.

\medskip
\noindent{\it Step 3.} Decay estimate. Let $\tilde{f}=e^{\lambda t}f$ with $\lambda>0$. Then the equation of $\tilde{f}$ reads as:
\begin{equation}\nonumber\left\{
\begin{aligned}
&\pa_t\tilde{f}+v_1\pa_{x}\tilde{f}+\sqrt{\mu}^{-1}G(t,x)\pa_{v_1}(\sqrt{\mu}\tilde{f})+L\tilde{f}=\lambda\tilde{f}+e^{\lambda t}g,\\
&\tilde{f}|_{\g_-}=P_{\g}\tilde{f}+e^{\lambda t}r,\\
&\tilde{f}(0,x,v)=f_0(x,v).\\
\end{aligned}\right.
\end{equation}
Notice that
$$\int_{0}^1\int_{\mathbb{R}^3}\tilde{f}\sqrt{\mu}\dd x\dd v=e^{\lambda t}\int_{0}^1\int_{\mathbb{R}^3}f\sqrt{\mu}\dd x\dd v=0.
$$
The applying \eqref{4.12} to $\tilde{f}$, we have
\begin{align}
\|e^{\lambda t}f(t)\|^2_{L^2}+\int_0^t\|\nu^{1/2}e^{\lambda s}f(s)\|^2_{L^2}\dd s\leq& C\|{f}_0\|^2_{L^2}+C\lambda^2\int_0^t\|e^{\lambda s}f(s)\|_{L^2}^2\dd s
\nonumber\\
&+C\int_{0}^t\|\nu^{-1/2}e^{\lambda s }g(s)\|_{L^2}^2+|e^{\lambda s}r(s)|_{L^2_-}^2\dd s.\nonumber
\end{align}
Taking $\lambda>0$ suitably small, we have
\begin{align}
\|f(t)\|_{L^2}^2\leq Ce^{-2\lambda t}\|f_0\|_{L^2}^2+C\int_{0}^te^{-2\lambda(t-s)}\{\|\nu^{-1/2}g(s)\|_{L^2}^2+|r(s)|_{L^2_-}^2\}\dd s,\nonumber
\end{align}
which implies, for any $0<\lambda_1<\lambda$, that
\begin{align}\label{4.15}
\|f(t)\|_{L^2}&\leq Ce^{-\lambda_1t}\|f_0\|_{L^2}+Ce^{-\lambda_1t}\sup_{0\leq s\leq t }\{\|\nu^{-1/2}e^{\lambda_1s}g(s)\|_{L^2}+|e^{\lambda_1s}r(s)|_{L^{2}_-}\}\nonumber\\
&\leq Ce^{-\lambda_1t}\|wf_0\|_{L^\infty}+Ce^{-\lambda_1t}\sup_{0\leq s\leq t}\{\|\nu^{-1}we^{\lambda_1s}g(s)\|_{L^\infty}+|we^{\lambda_1s}r(s)|_{L^{\infty}_-}\}.
\end{align}
Similarly, applying \eqref{4.12-1} to $e^{\lambda_1t}f$, we have
\begin{align}
&\sup_{0\leq s\leq t}\|e^{\lambda_1 s}wf(s)\|_{L^{\infty}}+\sup_{0\leq s\leq t}|e^{\lambda_1 s}wf(s)|_{L^{\infty}_\pm}\nonumber\\
&\quad\leq C \|wf_0\|_{L^{\infty}}+ C\sup_{0\leq s\leq t}\|\nu^{-1}we^{\lambda_1 s}g(s)\|_{L^\infty}+C\sup_{0\leq s\leq t}|we^{\lambda_1 s}r(s)|_{L^\infty_-}+C\sup_{0\leq s\leq t}\|e^{\lambda_1s}f(s)\|_{L^2}\nonumber\\
&\quad\leq C\|wf_0\|_{L^{\infty}}+ C\sup_{0\leq s\leq t}\|\nu^{-1}we^{\lambda_1 s}g(s)\|_{L^\infty}+C\sup_{0\leq s\leq t}|we^{\lambda_1 s}r(s)|_{L^\infty_-},\nonumber
\end{align}
where we have used \eqref{4.15} in the last inequality. Therefore, the estimate \eqref{4.5} follows and the proof of Proposition \ref{prop4.1} is completed.
\end{proof}

\noindent{\it Proof of Theorem \ref{thm1.1}:} The existence of the time-periodic solution of \eqref{1.2.2}, \eqref{1.2.3}, \eqref{mass1} has been proved in Theorem \ref{thm1.2}. Therefore, it suffices to show that $F\geq 0.$ To do this, we first prove that such a time-periodic solutions is exponentially stable under the dynamics of \eqref{4.1}, \eqref{4.2}, \eqref{4.2-1}. The global solution is constructed via the following iteration
\begin{equation}
\left\{
\begin{aligned}
&\pa_tf^{n+1}+v_1\pa_xf^{n+1}+\f1{\sqrt{\mu}}G(t,x)\pa_{v_1}(\sqrt{\mu}f^{n+1})+Lf^{n+1}=\Gamma(f^{\text{per}}+f^{n},f^{n})+\Gamma(f^{n},f^{\text{per}}),\\
&f^{n+1}|_{\g_-}=P_\g f^{n+1}+\frac{\sqrt{2\pi}(\mu_{w}-\mu)}{\sqrt{\mu}}\int_{v'\cdot n(x)>0}f^n\sqrt{\mu} \{v'\cdot n(x)\}\dd v',\\
&f^{n+1}(0,x,v)=f_0(x,v),\quad f^0\equiv0.\nonumber
\end{aligned}\right.
\end{equation}
A direct computation shows that
$$P_{\g}\left\{\frac{\sqrt{2\pi}(\mu_{w}-\mu)}{\sqrt{\mu}}\int_{v'\cdot n(x)>0}f^n\sqrt{\mu} \{v'\cdot n(x)\}\dd v'\right\}\equiv0,
$$
and
$$\left|w\left\{\frac{\sqrt{2\pi}(\mu_{w}-\mu)}{\sqrt{\mu}}\int_{v'\cdot n(x)>0}f^n\sqrt{\mu} \{v'\cdot n(x)\}\dd v'\right\}\right|_{L^{\infty}_-}\leq C\delta|f^{n}|_{L^{\infty}_+}.
$$
Then applying the linear decay estimate \eqref{4.5} to $f^{n+1}$ and using \eqref{g}, we have, for any $t>0$, that
\begin{align}
&\sup_{0\leq s\leq t}\|e^{\lambda_1s}wf^{n+1}(s)\|_{L^{\infty}}+\sup_{0\leq s\leq t}|we^{\lambda_1s}f^{n+1}(s)|_{L^{\infty}_\pm}\nonumber\\
&\quad\leq C \|wf_0\|_{L^{\infty}}+C\delta\sup_{0\leq s\leq t}\{\|we^{\lambda_1s}f^n(s)\|_{L^{\infty}}+|we^{\lambda_1s}f^n(s)|_{L^{\infty}_\pm}\}
+C\sup_{0\leq s\leq t}\|e^{\lambda_1 s}wf^n(s)\|_{L^{\infty}}^2.\nonumber
\end{align}
The same as in \eqref{3.3.7}, we can also use an inductive argument to show that
$$
\sup_{0\leq s\leq t}\|we^{\lambda_1s}f^{n+1}(s)\|_{L^{\infty}}+|we^{\lambda_1s}f^{n+1}(s)|_{L^{\infty}_\pm}\leq 2C \|wf_0\|_{L^{\infty}}, \quad\text{for}\quad j=0,1,\cdots,
$$
provided that both $\delta>0$ and $\|wf_0\|_{L^{\infty}}$ suitably small. Similarly, we can also show that $\{f^n\}_{n=1}^{\infty}$ is a Cauchy sequence. Then the solution is obtained by taking the limit $n\rightarrow+\infty$. It is straightforward to verify that the solution $f$ satisfies
\begin{equation}
\|wf(t)\|_{L^{\infty}}\leq Ce^{-\lambda_1t}\|wf_0\|_{L^{\infty}}.\nonumber
\end{equation}
Then the non-negativity of $F^{\text{per}}$ follows from the same argument as in \cite{EGKM2}. We omit here for brevity.
Notice that \eqref{1.1.3} follows from a change of coordinates in terms of \eqref{1.2.1}. Therefore, we have completed the proof of Theorem \ref{thm1.1}. $\hfill\Box$\\

\noindent{\bf Acknowledgments.}  Renjun Duan was partially supported by the General Research Fund (Project No.14301720) and the Direct Grant (Project No.4053213) from CUHK. Zhu Zhang started to work on this project during his Ph.D studies in The Chinese University of Hong Kong. He would like to thank the University for providing the great working environment and support.

\end{document}